\documentclass[a4paper,12pt]{article}

\usepackage{amsmath}
\usepackage{ntheorem}
\usepackage{amssymb}
\usepackage{latexsym}
\usepackage[all]{xy}\UseComputerModernTips
\usepackage{graphicx}

\theoremstyle{plain}
\newtheorem{proposition}{Proposition}[section]
\newtheorem{theorem}[proposition]{Theorem}

\newtheorem{corollary}[proposition]{Corollary}

\theoremstyle{plain}
\theorembodyfont{\normalfont\rm}
\newtheorem{definition}[proposition]{Definition}
\newtheorem{remark}[proposition]{Remark}
\newtheorem{example}[proposition]{Example}

\theoremstyle{nonumberplain}
\theoremheaderfont{\normalfont\itshape}
\theoremseparator{.}
\theorembodyfont{\normalfont}
\newtheorem{proof}{Proof}

\makeatletter

\@addtoreset{equation}{section}
\makeatother

\newcommand{\ZZ}{{\mathbb Z}}
\newcommand{\RR}{{\mathbb R}}
\newcommand{\CC}{{\mathbb C}}
\newcommand{\PP}{{\mathbb P}}
\newcommand{\LL}{{\mathbb L}}
\newcommand{\Hom}{{\rm Hom}}

\renewcommand{\d}{{\rm dim}}
\newcommand{\cd}{{\rm codim}}
\newcommand{\Vol}{{\rm Vol}}
\newcommand{\vol}{{\rm vol}}
\newcommand{\Eu}{{\rm Eu}}
\newcommand{\e}{\varepsilon}
\renewcommand{\SS}{{\mathcal S}}
\newcommand{\Spec}{{\rm Spec}}

\newcommand{\id}{{\rm id}}
\newcommand{\supp}{{\rm supp}}
\newcommand{\RSV}{{\rm RSV}}

\newcommand{\grad}{{\rm grad}}

\newcommand{\Db}{{\bf D}^{b}}
\newcommand{\Dbc}{{\bf D}_{c}^{b}}

\newcommand{\D}{{\cal D}}
\newcommand{\F}{{\cal F}}
\newcommand{\G}{{\cal G}}

\renewcommand{\L}{{\mathcal L}}
\newcommand{\M}{{\mathcal M}}
\renewcommand{\O}{{\cal O}}

\newcommand{\CF}{{\rm CF}}

\newcommand{\reg}{{\rm reg}}

\newcommand{\RG}{R\varGamma}
\newcommand{\Rhom}{R{\mathcal Hom}}

\newcommand{\tl}[1]{\widetilde{#1}}

\newcommand{\simto}{\overset{\sim}{\longrightarrow}}

\newcommand{\dsum}{\displaystyle \sum}

\renewcommand{\(}{\left(}
\renewcommand{\)}{\right)}
\newcommand{\inun}{\text{\rotatebox[origin=c]{90}{$\in$}}}
\newcommand{\longhookrightarrow}{\DOTSB\lhook\joinrel\longrightarrow}
\newcommand{\longtwoheadrightarrow}{\relbar\joinrel\twoheadrightarrow}

\renewcommand{\labelenumi}{{\rm (\roman{enumi})}}

\title{A geometric degree formula for $A$-discriminants and Euler obstructions of toric varieties}
\author{Yutaka \textsc{Matsui}\footnote{Department of Mathematics, Kinki University, 3-4-1, Kowakae, Higashi-Osaka, Osaka, 577-8502, Japan.} \and Kiyoshi \textsc{Takeuchi}\footnote{Institute of Mathematics, University  of Tsukuba, 1-1-1, Tennodai, Tsukuba, Ibaraki, 305-8571, Japan.}}

\date{}

\begin{document}

\maketitle

\begin{abstract}
We give explicit formulas for the dimensions and the degrees of $A$-discriminant varieties introduced by Gelfand-Kapranov-Zelevinsky \cite{G-K-Z}. Our formulas can be applied also to the case where the $A$-discriminant varieties are higher-codimensional and their degrees are described by the geometry of the configurations $A$. Moreover combinatorial formulas for the Euler obstructions of general (not necessarily normal) toric varieties will be also given.
\end{abstract}

\section{Introduction}\label{sec:1}

The theory of discriminants is on the crossroad of various branches of mathematics, such as commutative algebra, algebraic geometry, singularity theory and topology. In \cite{G-K-Z}, Gelfand-Kapranov-Zelevinsky generalized this classical theory to polynomials of several variables by introducing $A$-discriminant varieties and obtained many deep results. They thus laid the foundation of the modern theory of discriminants. The first aim of this paper is to give formulas for the dimensions and the degrees of $A$-discriminant varieties. Let $A$ be a finite subset a lattice $M=\ZZ^n$. Then the $A$-discriminant variety $X_A^*$ is the dual of a projective variety $X_A \subset \PP^{\sharp A-1}$ over $\CC$ defined by $A$ (see \cite{G-K-Z} and Section \ref{sec:2} for the definition). Let $P$ be the convex hull of $A$ in $M_{\RR}=\RR \otimes_{\ZZ} M$ and assume that $\d P=n$. For each face $\Delta \prec P$ of the polytope $P$, we denote by $\Vol_{\ZZ}(\Delta) \in \ZZ$ the normalized ($\d \Delta$)-dimensional volume (see \eqref{eq:3-3}) of $\Delta$ with respect to the affine sublattice $M(A \cap \Delta)$ of $M$ generated by $A \cap \Delta$. Recall that the algebraic torus $T=\Spec(\CC[M]) \simeq (\CC^*)^n$ naturally acts on $X_A$ with finitely many orbits and there exists a basic correspondence ($0 \leq k \leq n=\d P$):
\begin{equation}
\{ \text{$k$-dimensional faces of $P$}\} \overset{\text{1:1}}{\longleftrightarrow} \{ \text{$k$-dimensional $T$-orbits in $X_A$} \}
\end{equation}
proved by \cite[Chapter 5, Proposition 1.9]{G-K-Z}. For a face $\Delta \prec P$ of $P$, we denote by $T_{\Delta}$ the $T$-orbit in $X_A$ which corresponds to $\Delta$. Then our first result is as follows.

\begin{theorem}\label{thm:1-1}
For $1 \leq i \leq \sharp A-1$, set
\begin{equation}
\delta_i :=\sum_{\Delta \prec P}(-1)^{\cd \Delta} \left\{\binom{\d\Delta-1}{i} +(-1)^{i-1}(i+1)\right\}\Vol_{\ZZ}(\Delta) \cdot \Eu(\Delta),
\end{equation}
where $\Eu(\Delta)$ is the value of the Euler obstruction $\Eu_{X_A} \colon X_A \longrightarrow \ZZ$ of $X_A$ on the $T$-orbit $T_{\Delta} \simeq (\CC^*)^{\d \Delta}\subset X_A$. For the meaning of the binomial coefficient $\binom{\d\Delta-1}{i}$, see Remark \ref{rem:2-5} (i). Then the codimension $r=\cd X_A^*=\sharp A-1-\d X_A^*$ and the degree of the $A$-discriminant variety $X_A^*$ are given by
\begin{eqnarray}
r= \cd X_A^* &=& \min \{i \ |\ \delta_i \neq 0\},\\
\deg X_A^* &=& \delta_r.
\end{eqnarray}
\end{theorem}

The above theorem will be proved by using Ernstr{\"o}m's degree formula for dual varieties in \cite{Ernstrom} and a result in \cite{M-T}. Note that very recently by using tropical algebraic geometry, also Dickenstein-Feichtner-Sturmfels \cite{D-F-S} obtained a degree formula for the $A$-discriminant variety $X_A^*$ when $X_A^*$ is a hypersurface. Our formula is applicable also to the case where $X_A^*$ is higher-codimensional. Moreover, our formula is more directly related to the geometry of the convex polytope $P$. In particular, if $X_A$ is smooth, our formula coincides with Gelfand-Kapranov-Zelevinsky's theorem \cite[Chapter 9, Theorem 2.8]{G-K-Z}. In fact, if $X_A^*$ is a hypersurface the above formula is implicit in their prime factorization theorem \cite[Chapter 10, Theorem 1.2]{G-K-Z}, although they did not formulate it in terms of Euler obstructions. In Section \ref{sec:4}, we will give combinatorial formulas for the Euler obstruction $\Eu_{X_A} \colon X_A \longrightarrow \ZZ$ of $X_A$. Combining them with Theorem \ref{thm:1-1} above, we can now calculate the dimension and the degree of $X_A^*$ for any configuration $A \subset M=\ZZ^n$ (see Example \ref{exa:6-8} in Section \ref{sec:4}). Recently in \cite{Esterov} Esterov found some nice applications of our formula. 

Our functorial proof of the formula for the Euler obstruction $\Eu_{X_A} \colon X_A \longrightarrow \ZZ$ leads us to various applications. In Section \ref{sec:7}, we derive from it useful formulas (Theorem \ref{thm:7-3} and \ref{thm:7-4}) for the characteristic cycles of $T$-equivariant constructible sheaves on general (not necessarily normal) toric varieties. See \cite{Braden} for another approach to this problem. In particular, combining it with the combinatorial description of the intersection cohomology complexes of toric varieties obtained by Bernstein-Khovanskii-MacPherson (unpublished), Denef-Loeser \cite{D-L} and Fieseler \cite{Fieseler} etc., we can derive combinatorial formulas for the characteristic cycle of the intersection cohomology complex of any normal toric variety. See Section \ref{sec:7} for the detail. Note that in \cite[Chapter 10, Theorem 2.11]{G-K-Z} also Gelfand-Kapranov-Zelevinsky obtained a formula for the characteristic cycles in a special but important case, from which they could have obtained the same result by some generalization (or a reformulation). However we included here a proof of Theorem \ref{thm:7-3} and \ref{thm:7-4}, since we can not find such an explicit presentation in the literature. We hope that we could show the power and the beauty of the sheaf-theoretical methods (see for example, \cite{Dimca}, \cite{H-T-T} and \cite{K-S}) by proving them functorially. Also, in the proof of Theorem \ref{thm:6-7}, we gave an explicit description \eqref{eq:4-53} of the branches along torus orbits in non-normal toric varieties found in \cite[Chapter 5, Theorem 3.1]{G-K-Z} and clarified the treatment of non-normal toric varieties in \cite[Chapter 5]{G-K-Z}. This result would have many applications in the future. Finally, let us mention that combining our combinatorial description of the Euler obstructions of toric varieties with the result of Ehlers (unpublished) and Barthel-Brasselet-Fieseler \cite{B-B-F} we can now compute the Chern-Mather classes of complete toric varieties very easily.

\bigskip
\noindent{\bf Acknowledgement:} The authors would like to thank Professor Esterov for some useful discussions.

\section{Preliminary notions and results}\label{sec:pre}

In this section, we introduce basic notions and results which will be used in this paper. In this paper, we essentially follow the terminology of \cite{Dimca}, \cite{H-T-T} and \cite{K-S}. For example, for a topological space $X$ we denote by $\Db(X)$ the derived category whose objects are bounded complexes of sheaves of $\CC_X$-modules on $X$.

\begin{definition}
Let $X$ be an algebraic variety over $\CC$. Then
\begin{enumerate}
\item We say that a sheaf $\F$ on $X$ is constructible if there exists a stratification $X=\bigsqcup_{\alpha} X_{\alpha}$ of $X$ such that $\F|_{X_{\alpha}}$ is a locally constant sheaf of finite rank for any $\alpha$.
\item We say that an object $\F$ of $\Db(X)$ is constructible if the cohomology sheaf $H^j(\F)$ of $\F$ is constructible for any $j \in \ZZ$. We denote by $\Dbc(X)$ the full subcategory of $\Db(X)$ consisting of constructible objects $\F$.
\end{enumerate}
\end{definition}

Recall that for any morphism $f \colon X \longrightarrow Y$ of algebraic varieties over $\CC$ there exists a functor
\begin{equation}
Rf_* \colon \Db(X) \longrightarrow \Db(Y)
\end{equation}
of direct images. This functor preserves the constructibility and we obtain also a functor
\begin{equation}
Rf_* \colon \Dbc(X) \longrightarrow \Dbc(Y).
\end{equation}
For other basic operations $Rf_!$, $f^{-1}$, $f^!$ etc. in derived categories, see \cite{K-S} for the detail.

Next we introduce the notion of constructible functions. 

\begin{definition}
Let $X$ be an algebraic variety over $\CC$. Then we say a $\ZZ$-valued function $\rho \colon X \longrightarrow \ZZ$ on $X$ is constructible if there exists a stratification $X=\bigsqcup_{\alpha} X_{\alpha}$ of $X$ such that $\rho|_{X_{\alpha}}$ is constant for any $\alpha$. We denote by $\CF_{\ZZ}(X)$ the abelian group of constructible functions on $X$.
\end{definition}

For a constructible function $\rho \colon X \longrightarrow \ZZ$ take a stratification $X=\bigsqcup_{\alpha}X_{\alpha}$ of $X$ such that $\rho|_{X_{\alpha}}$ is constant for any $\alpha$ as above. Denoting the Euler characteristic of $X_{\alpha}$ by $\chi(X_{\alpha})$ we set
\begin{equation}
\int_X \rho :=\dsum_{\alpha}\chi(X_{\alpha}) \cdot \rho(x_{\alpha}) \in \ZZ,
\end{equation}
where $x_{\alpha}$ is a reference point in $X_{\alpha}$. Then we can easily show that $\int_X\rho \in \ZZ$ does not depend on the choice of the stratification $X=\bigsqcup_{\alpha} X_{\alpha}$ of $X$. We call $\int_X \rho \in \ZZ$ the topological (Euler) integral of $\rho$ over $X$.

Among various operations in derived categories, the following nearby and vanishing cycle functors introduced by Deligne will be frequently used in this paper (see \cite[Section 4.2]{Dimca} for an excellent survey of this subject).

\begin{definition}
Let $f \colon X \longrightarrow \CC$ be a non-constant regular function on an algebraic variety $X$ over $\CC$. Set $X_0:= \{x\in X\ |\ f(x)=0\} \subset X$ and let $i_X \colon X_0 \longhookrightarrow X$, $j_X \colon X \setminus X_0 \longhookrightarrow X$ be inclusions. Let $p \colon \tl{\CC^*} \longrightarrow \CC^*$ be the universal covering of $\CC^* =\CC \setminus \{0\}$ ($\tl{\CC^*} \simeq \CC$) and consider the Cartesian square
\begin{equation}\label{eq:pre-1}
\xymatrix@R=2.5mm@C=2.5mm{
\tl{X \setminus X_0} \ar[rr] \ar[dd]^{p_X} & &\tl{\CC^*} \ar[dd]^p \\
 & \Box & \\
X \setminus X_0 \ar[rr]^f & & \CC^*.}
\end{equation}
Then for $\F \in \Db(X)$ we set
\begin{equation}
\psi_f(\F) := i_X^{-1}R(j_X \circ p_X)_*(j_X \circ p_X)^{-1}\F \in \Db(X_0)
\end{equation}
and call it the nearby cycle of $\F$. We also define the vanishing cycle $\varphi_f(\F)\in \Db(X_0)$ of $\F$ to be the third term of the distinguished triangle:
\begin{equation}
i_X^{-1}\F \overset{{\rm can}}{\longrightarrow} \psi_f(\F) \longrightarrow \varphi_f(\F) \overset{+1}{\longrightarrow}
\end{equation}
in $\Db(X_0)$, where ${\rm can} \colon i_X^{-1}\F \longrightarrow \psi_f(\F)$ is the canonical morphism induced by $\id \longrightarrow R(j_X \circ p_X)_*(j_X \circ p_X)^{-1}$.
\end{definition}

Since nearby and vanishing cycle functors preserve the constructibility, in the above situation we obtain functors
\begin{equation}
\psi_f, \ \varphi_f \colon \Dbc(X) \longrightarrow \Dbc(X_0).
\end{equation}

The following theorem will play a crucial role in this paper. For the proof, see for example, \cite[Proposition 4.2.11]{Dimca}.

\begin{theorem}\label{thm:pre-10}
Let $\pi \colon Y \longrightarrow X$ be a proper morphism of algebraic varieties over $\CC$ and $f \colon X \longrightarrow \CC$ a non-constant regular function on $X$. Set $g:=f \circ \pi \colon Y \longrightarrow \CC$, $X_0:=\{x\in X\ |\ f(x)=0\}$ and $Y_0:=\{y\in Y\ |\ g(y)=0\}=\pi^{-1}(X_0)$. Then for any $\G\in \Dbc(Y)$ we have
\begin{eqnarray}
R (\pi |_{Y_0})_* \psi_g(\G) &=&\psi_f(R\pi_*\G),\\
R (\pi |_{Y_0})_* \varphi_g(\G) &=& \varphi_f(R\pi_*\G)
\end{eqnarray}
in $\Dbc(X_0)$, where $ \pi|_{Y_0} \colon Y_0 \longrightarrow X_0$ is the restriction of $\pi$.
\end{theorem}

Let us recall the definition of characteristic cycles of constructible sheaves. Let $X$ be a smooth algebraic variety over $\CC$ and $\F \in \Dbc(X)$. Then there exists a Whitney stratification $X=\bigsqcup_{\alpha} X_{\alpha}$ of $X$ consisting of connected strata $X_{\alpha}$ such that $H^j(\F)|_{X_{\alpha}}$ is a locally constant sheaf for any $j \in \ZZ$ and $\alpha$. For a point $x_{\alpha} \in X_{\alpha}$, take a holomorphic function $f \colon U_{\alpha} \longrightarrow \CC$ defined in a neighborhood $U_{\alpha}$ of $x_{\alpha}$ in $X$ which satisfies the conditions
\begin{enumerate}
\item $f(x_{\alpha})=0$,
\item $(x_{\alpha};\grad f(x_{\alpha}))\in T_{X_{\alpha}}^*X \setminus \(\bigcup_{\beta \neq \alpha}\overline{T_{X_{\beta}}^*X}\)$,
\item $x_{\alpha} \in X_{\alpha}$ is a non-degenerate critical point of $f|_{X_{\alpha}}$,
\end{enumerate}
and set
\begin{eqnarray}
m_{\alpha}
&:=& -\chi(\varphi_f(\F)_{x_{\alpha}}) \\
&=& -\dsum_{j \in \ZZ} (-1)^j \d_{\CC}\(H^j(\varphi_f(\F))_{x_{\alpha}}\) \in \ZZ.
\end{eqnarray}
Then we can show that the integer $m_{\alpha}$ does not depend on the choice of the stratification $X=\bigsqcup_{\alpha} X_{\alpha}$, $x_{\alpha}\in X_{\alpha}$ and $f$.

\begin{definition}
By using the above integers $m_{\alpha} \in \ZZ$, we define a Lagrangian cycle $CC(\F)$ in the cotangent bundle $T^*X$ of $X$ by
\begin{equation}
CC(\F) :=\dsum_{\alpha}m_{\alpha} \left[ \overline{T_{X_{\alpha}}^*X} \right].
\end{equation}
We call $CC(\F)$ the characteristic cycle of $\F \in \Dbc(X)$. Its coefficient $m_{\alpha} \in \ZZ$ is called the multiplicity of $\F$ along the Lagrangian subvariety $\overline{T_{X_{\alpha}}^*X} \subset T^*X$.
\end{definition}

Recall that in $\Dbc(X)$ there exists a full abelian subcategory ${\rm Perv}(X)$ of perverse sheaves (see \cite{H-T-T} and \cite{K-S} etc. for the detail of this subject). Although for the definition of perverse sheaves there are some different conventions of shifts in the literature, here we adopt the one in \cite{H-T-T} by which the shifted constant sheaf $\CC_X[\d X]\in \Dbc(X)$ on a smooth algebraic variety $X$ is perverse. Then for any perverse sheaf $\F \in {\rm Perv}(X)\subset \Dbc(X)$ on a smooth algebraic variety $X$ we can easily show that the multiplicities in the characteristic cycle $CC(\F)$ of $\F$ are non-negative.

\begin{example}
Let $X=\CC_x^n$ and $Y=\{x_1=\cdots =x_d=0\} \subset X=\CC_x^n$. Set $\F:=\CC_Y[n-d] \in {\rm Perv}(X)$. Then by an easy computation
\begin{equation}
m=-\chi(\varphi_f(\CC_Y[n-d])_0)=1
\end{equation}
for $f(x)=x_1+x_{d+1}^2+\cdots +x_n^2$ at $0 \in Y\subset X=\CC_x^n$ we obtain
\begin{equation}
CC(\F)=1 \cdot [T_Y^*X].
\end{equation}
\end{example}

Finally, we recall a special case of Bernstein-Khovanskii-Kushnirenko's theorem \cite{Khovanskii}. 

\begin{definition}
Let $g(x)=\sum_{v \in \ZZ^n} a_vx^v$ be a Laurent polynomial on $(\CC^*)^n$ ($a_v\in \CC$). We call the convex hull of $\supp(g):=\{v\in \ZZ^n \ |\ a_v\neq 0\} \subset \ZZ^n \subset \RR^n$ in $\RR^n$ the Newton polygon of $g$ and denote it by $NP(g)$.
\end{definition}

\begin{theorem}[\cite{Khovanskii}]\label{thm:2-10}
Let $\Delta$ be an integral polytope in $\RR^n$ and $g_1, \ldots, g_p$ generic Laurent polynomials on $(\CC^*)^n$ satisfying $NP(g_i)=\Delta$. Then the Euler characteristic of the subvariety $Z^*=\{ x\in (\CC^*)^n \ |\ g_1(x)=\cdots =g_p(x)=0 \}$ of $(\CC^*)^n$ is given by
\begin{equation}
\chi(Z^*)=(-1)^{n-p}\binom{n-1}{p-1}\Vol_{\ZZ}(\Delta),
\end{equation}
where $\Vol_{\ZZ}(\Delta)\in \ZZ$ is the normalized $n$-dimensional volume of $\Delta$ with respect to the lattice $\ZZ^n \subset \RR^n$.
\end{theorem}

\section{Degree formulas for $A$-discriminant varieties}\label{sec:2}
In this section, we first introduce the formula for the degrees of $A$-discriminants obtained by Gelfand-Kapranov-Zelevinsky \cite{G-K-Z} and prove our generalization.

Let $M \simeq \ZZ^n$ be a $\ZZ$-lattice (free $\ZZ$-module) of rank $n$ and $M_{\RR}:=\RR \otimes_{\ZZ}M$ the real vector space associated with $M$. Let $A \subset M$ be a finite subset of $M$ and denote by $P$ its convex hull in $M_{\RR}$. In this paper, such a polytope $P$ will be called an integral polytope in $M_{\RR}$. If $A=\{\alpha(1),\alpha(2),\ldots, \alpha(m+1)\}$, we can define a morphism $\Phi_A \colon T\longrightarrow \PP^m$ ($m:=\sharp A-1$) from an algebraic torus $T:=\Spec (\CC [M])=(\CC^*)^n$ to a complex projective space $\PP^m$ by 
\begin{equation}
x=(x_1,x_2,\ldots,x_n) \longmapsto [x^{\alpha(1)} \colon x^{\alpha(2)} \colon \cdots \colon x^{\alpha(m+1)}],
\end{equation}
where for each $\alpha(i) \in A \subset M\simeq \ZZ^n$ and $x \in T$ we set $x^{\alpha(i)}=x_1^{\alpha(i)_1}x_2^{\alpha(i)_2}\cdots x_n^{\alpha(i)_n}$ as usual.

\begin{definition}[\cite{G-K-Z}]
Let $X_A:=\overline{{\rm im}\ \Phi_A}$ be the closure of the image of $\Phi_A \colon T \longrightarrow \PP^m$. Then the dual variety $X_A^* \subset (\PP^m)^*$ of $X_A$ is called the $A$-discriminant variety. If moreover $X_A^*$ is a hypersurface in the dual projective space $(\PP^m)^*$, then the defining homogeneous polynomial of $X_A^*$ (which is defined up to non-zero constant multiples) is called the $A$-discriminant.
\end{definition}

Note that the $A$-discriminant variety $X_A^*$ is naturally identified with the set of Laurent polynomials $f \colon T=(\CC^*)^n \longrightarrow \CC$ of the form $f(x)=\sum_{\alpha \in A} a_{\alpha} x^{\alpha}$ ($a_{\alpha} \in \CC$) such that $\{x\in T\ |\ f(x)=0 \}$ is a singular hypersurface in $T$. In order to introduce the degree formula for $A$-discriminants proved by Gelfand-Kapranov-Zelevinsky \cite{G-K-Z}, we need the following.

\begin{definition}[\cite{G-K-Z}]
For a subset $B \subset M \simeq \ZZ^n$, we define an affine $\ZZ$-sublattice $M(B)$ of $M$ by
\begin{equation}
M(B):=\left\{ \sum_{v \in B} c_v \cdot v \ \left| \ c_v \in \ZZ, \ \sum_{v \in B} c_v=1 \right.\right\}.
\end{equation}
\end{definition}

Let $\Delta \prec P$ be a face of $P$ and denote by $\LL(\Delta)$ the smallest affine subspace of $M_{\RR}$ containing $\Delta$. Then $M(A \cap \Delta )$ is a $\ZZ$-lattice of rank $\d \Delta=\d \LL(\Delta)$ in $\LL(\Delta)$ and we have $(M(A \cap \Delta))_{\RR}\simeq \LL(\Delta)$. Let $\vol$ be the Lebesgue measure of $(\LL(\Delta), M(A \cap \Delta))$ by which the volume of the fundamental domain of the translation by $M(A \cap \Delta)$ on $\LL(\Delta)$ is measured to be $1$. For a subset $K \subset \LL(\Delta)$, we set
\begin{equation}\label{eq:3-3}
\Vol_{\ZZ}(K):=(\d \Delta)! \cdot \vol(K).
\end{equation}
We call it the normalized $(\d \Delta)$-dimensional volume of $K$ with respect to the lattice $M(A\cap \Delta)$. Throughout this paper, we use this normalized volume $\Vol_{\ZZ}$ instead of the usual one.

The following formula is obtained by Gelfand-Kapranov-Zelevinsky \cite[Chapter 9, Theorem 2.8]{G-K-Z}.

\begin{theorem}[\cite{G-K-Z}]\label{thm:2-3}
Assume that $X_A \subset \PP^m$ is smooth and $X_A^*$ is a hypersurface in $(\PP^m)^*$. Then the degree of the $A$-discriminant is given by the formula:
\begin{equation}
\deg X_A^* =\sum_{\Delta \prec P} (-1)^{\cd \Delta}(\d \Delta +1) \Vol_{\ZZ}(\Delta).
\end{equation}
\end{theorem}

In order to state our generalization of Theorem \ref{thm:2-3} to the case where $X_A^*$ may be higher-codimensional, recall that $T=\Spec (\CC [M])$ acts naturally on $X_A$ and we have a basic correspondence ($0 \leq k \leq n=\d P$):
\begin{equation}
\{ \text{$k$-dimensional faces of $P$}\} \overset{\text{1:1}}{\longleftrightarrow} \{ \text{$k$-dimensional $T$-orbits in $X_A$} \}
\end{equation}
proved by \cite[Chapter 5, Proposition 1.9]{G-K-Z}. For a face $\Delta \prec P$ of $P$, we denote by $T_{\Delta}$ the corresponding $T$-orbit in $X_{A}$. We denote the value of the Euler obstruction $\Eu_{X_A} \colon X_A \longrightarrow \ZZ$ of $X_A$ on $T_{\Delta}$ by $\Eu(\Delta) \in \ZZ$. The precise definition of the Euler obstruction will be given later in Section \ref{sec:4}. Here we simply recall that the Euler obstruction of $X_A$ is constant along each $T$-orbit $T_{\Delta}$ and takes the value $1$ on the smooth part of $X_A$. In particular, for $\Delta=P$ the $T$-orbit $T_{\Delta}$ is open dense in $X_A$ and $\Eu(\Delta)=1$.

\begin{theorem}\label{thm:2-4}
For $1 \leq i \leq m$, set
 \begin{equation}\label{eq:2-9}
\delta_i :=\sum_{\Delta \prec P}(-1)^{\cd \Delta} \left\{\binom{\d \Delta-1}{i} +(-1)^{i-1}(i+1)\right\}\Vol_{\ZZ}(\Delta) \cdot \Eu(\Delta).
\end{equation}
Then the codimension $r=\cd X_A^*=m-\d X_A^*$ and the degree of the dual variety $X_A^*$ are given by
\begin{eqnarray}
r= \cd X_A^* &=& \min \{ i \ |\ \delta_i \neq 0\},\\
\deg X_A^* &=& \delta_r.
\end{eqnarray}
\end{theorem}

\begin{remark}\label{rem:2-5}
\begin{enumerate}
\item For $p \in \ZZ$ and $q\in \ZZ_{\geq 0}$, we used the generalized binomial coefficient
\begin{equation}
\binom{p}{q}=\dfrac{p(p-1)(p-2)\cdots (p-q+1)}{q!}.
\end{equation}
For example, for a vertex $\Delta=\{v\}\prec P$, we have $\binom{\d \Delta-1}{i}=\binom{-1}{i}=(-1)^i$.
\item Note that the number $\cd X_A^*-1$ is called the dual defect of $X_A$.
\end{enumerate}
\end{remark}

\begin{proof}
First, by \cite[Chapter 5, Proposition 1.2]{G-K-Z} we may assume that $M(A)=M$. Recall that for $\alpha(j) \in A$ ($1 \leq j \leq \sharp A= m+1$) the function
\begin{equation}
T=(\CC^*)^n \ni x=(x_1,x_2,\ldots, x_n) \longmapsto x^{\alpha(j)} \in \CC^*
\end{equation}
is defined by the canonical pairing
\begin{equation}
T\times M=\Hom_{\ZZ}(M, \CC^*) \times M \longrightarrow \CC^*,
\end{equation}
where we consider $\CC^*$ as an abelian group (i.e. a $\ZZ$-module) and $\Hom_{\ZZ}(M, \CC^*)$ denotes the group of homomorphisms of $\ZZ$-modules from $M$ to $\CC^*$. Let us consider an affine chart
\begin{equation}
U_{m+1}:=\{ [\xi_1:\xi_2:\cdots :\xi_{m+1}] \in \PP^m \ |\ \xi_{m+1} \neq 0\}\simeq \CC^m
\end{equation}
of $\PP^m$. Then for any $x \in T=\Hom_{\ZZ}(M,\CC^*)$ we have $x^{\alpha(m+1)} \neq 0$ and there exists a morphism
\begin{equation}
\begin{array}{ccrcc}
\Psi_A &\colon & T & \longrightarrow &U_{m+1} \simeq \CC^m\\
& & \inun & &\inun \\
& & x &\longmapsto & (x^{\alpha(1)-\alpha(m+1)},\ldots, x^{\alpha(m)-\alpha(m+1)})
\end{array}
\end{equation}
induced by $\Phi_A \colon T \longrightarrow \PP^m$. Therefore, for $x,y\in T$ we have
\begin{eqnarray}
\Phi_A(x)=\Phi_A(y)
&\Longleftrightarrow& \Psi_A(x)=\Psi_A(y)\\
&\Longleftrightarrow& x^{\alpha}=y^{\alpha} \hspace{3mm}\text{for any $\alpha \in \dsum_{j=1}^m \ZZ(\alpha(j)-\alpha(m+1))$.}
\end{eqnarray}
If we take $\alpha(m+1)\in A$ to be the origin of the lattices $M$ and $M(A)$, then we obtain an isomorphism $M \simeq M(A)= \dsum_{j=1}^m \ZZ(\alpha(j)-\alpha(m+1))$ of lattices ($\ZZ$-modules). Note that the morphism $\Phi_A \colon T \longrightarrow \PP^m$ is not changed by this change of the origin of $M=M(A)$. Therefore we see that the morphism $\Phi_A \colon T \simeq (\CC^*)^n \longrightarrow \PP^m$ induces an isomorphism $T\simeq \Phi_A(T)\simeq (\CC^*)^n$. Note that $\Phi_A(T) \simeq (\CC^*)^n$ is the largest $T$-orbit $T_P$ in $X_A=\overline{{\rm im}\Phi_A} \subset \PP^m$. We can construct such an isomorphism also for any $T$-orbit $T_{\Delta}$ ($\Delta \prec P$) in $X_A$ as follows. For a face $\Delta \prec P$ of $P$, taking a point $\alpha(j) \in M(A \cap \Delta) \subset M \cap \LL(\Delta)$ to be the origin of the lattices $M(A \cap \Delta)$ and $M \cap \LL(\Delta)$, we consider $M(A \cap \Delta)$ as a sublattice ($\ZZ$-submodule) of $M \cap \LL(\Delta)$. By this choice of the origin $0 =\alpha(j)$ of the lattice $M(A \cap \Delta) \simeq \ZZ^{\d \Delta}$, we can construct a morphism $\Phi_{A \cap \Delta} \colon \Hom_{\ZZ}(M(A \cap \Delta) ,\CC^*) \simeq (\CC^*)^{\d \Delta} \longrightarrow \PP^m$ as follows. First, for $x \in \Hom_{\ZZ}(M(A \cap \Delta),\CC^*) \simeq (\CC^*)^{\d \Delta}$ and $\alpha \in M(A \cap \Delta)$ denote by $x^{\alpha}\in \CC^*$ the image of the pair $(x,\alpha)$ by the canonical paring
\begin{equation}
\Hom_{\ZZ}(M(A \cap \Delta),\CC^*)\times (M(A \cap\Delta)) \longrightarrow \CC^*.
\end{equation}
Then the morphism $\Phi_{A \cap \Delta} \colon\Hom_{\ZZ}(M(A \cap \Delta), \CC^*) \longrightarrow \PP^m$ is defined by
\begin{equation}
\Phi_{A \cap \Delta}(x)=[\xi_1:\xi_2:\cdots :\xi_{m+1}]
\end{equation}
for $x\in \Hom_{\ZZ}(M(A \cap \Delta),\CC^*) \simeq (\CC^*)^{\d \Delta}$, where we set
\begin{equation}
\xi_k :=\begin{cases}x^{\alpha(k)} & \text{if $\alpha(k) \in A \cap \LL(\Delta)$}, \\
0 & \text{otherwise}.
\end{cases}
\end{equation}
In this situation, by \cite[Proposition 1.2 and Proposition 1.9 in Chapter 5]{G-K-Z} the $T$-orbit $T_{\Delta}$ coincides with the image of $\Phi_{A \cap \Delta}$ and we can similarly prove that the morphism
\begin{equation}
\Phi_{A \cap \Delta} \colon \Hom_{\ZZ}(M(A \cap \Delta),\CC^*) \longtwoheadrightarrow \Phi_{A \cap \Delta}((\CC^*)^{\d \Delta})=T_{\Delta}
\end{equation}
is an isomorphism. By making use of this very simple description of \\$\Phi_{A\cap \Delta} \colon (\CC^*)^{\d \Delta} \simto T_{\Delta}$ for faces $\Delta \prec P$, we can now give a proof of our theorem. For $1 \leq i \leq m$, we take a generic linear subspace $H \simeq \PP^{m-1}$ (resp. $H_{i+1} \simeq \PP^{m-i-1}$) of $\PP^m$ of codimension $1$ (resp. $i+1$) and set
\begin{equation}\label{eq:2-4-1}
\delta_i:=(-1)^{n+i-1}\left\{ i 
\int_{\PP^m}\Eu_{X_A} -(i+1) \int_H 
\Eu_{X_A} +\int_{H_{i+1}}\Eu_{X_A} \right\}.
\end{equation}
Here we set $H_{m+1}:=\emptyset$. Then by \cite[Theorem 1.1]{Ernstrom} and \cite[Remark 3.3]{M-T} (see also \cite{M-T-2} and \cite{M-T-1}) the codimension $r=\cd X_A^*=m-\d X_A^*$ and the degree of the dual variety $X_A^*\subset (\PP^m)^*$ of $X_A$ are given by
\begin{gather}
r=\cd X_A^*=\min\{1\leq i \leq m\ |\ \delta_i \neq 0\},\\
\deg X_A^* =\delta_r.
\end{gather}
Hence it remains for us to rewrite the above integers $\delta_i$ ($1 \leq i \leq m$). First of all, since the Euler obstruction $\Eu_{X_A} \colon X_A \longrightarrow \ZZ$ is constant on each $T$-orbit $T_{\Delta} \simeq (\CC^*)^{\d \Delta}$ for $\Delta \prec P$ and $\chi((\CC^*)^d)=0$ for $d\geq 1$, we have
\begin{equation}\label{eq:2-4-2}
\int_{\PP^m}\Eu_{X_A} =\dsum_{\begin{subarray}{c} \Delta \prec P\\ \d \Delta=0\end{subarray}} \Eu(\Delta).
\end{equation}
Next, by taking a generic hyperplane
\begin{equation}
H=\left\{ [\xi_1:\xi_2:\cdots :\xi_{m+1}]\in \PP^m \ \left|\ \dsum_{j=1}^{m+1} a_j\xi_j=0 \right.\right\}
\end{equation}
($a_j \in \CC$) of $\PP^m$, we can calculate the topological integral $\int_H \Eu_{X_A}$ as follows. Since $\Phi_{A \cap \Delta} \colon \Hom_{\ZZ}(M(A \cap \Delta),\CC^*)\simeq (\CC^*)^{\d \Delta} \longrightarrow T_{\Delta}$ is an isomorphism, for the Laurent polynomial
\begin{equation}
\begin{array}{ccccc}
L_{\Delta} & \colon & \Hom_{\ZZ}(M(A \cap \Delta), \CC^*) & \longrightarrow & \CC^* \\
& & \inun & & \inun \\
& & x & \longmapsto & \dsum_{\alpha(j) \in A \cap \Delta} a_j x^{\alpha(j)}
\end{array}
\end{equation}
on the torus $\Hom_{\ZZ}(M(A \cap \Delta), \CC^*) \simeq (\CC^*)^{\d \Delta}$ we have
\begin{equation}
\chi(T_{\Delta} \cap H)=\chi(\{x\in (\CC^*)^{\d \Delta} \ |\ L_{\Delta}(x)=0\}).
\end{equation}
Note that for a generic hyperplane $H \subset \PP^m$ the hypersurface $\{x\in (\CC^*)^{\d \Delta} \ |\ L_{\Delta}(x)=0\}$ in the torus $\Hom_{\ZZ}(M(A \cap \Delta),\CC^*) \simeq (\CC^*)^{\d \Delta}$ cut out by $H$ satisfies the assumption of Bernstein-Khovanskii-Kushnirenko's theorem (Theorem \ref{thm:2-10}) for any $\Delta \prec P$. By Theorem \ref{thm:2-10}, we thus obtain
\begin{equation}\label{eq:2-4-3}
\int_H \Eu_{X_A}=\dsum_{\begin{subarray}{c}\Delta \prec P\\ \d \Delta \geq 1\end{subarray}}(-1)^{\d \Delta -1} \Vol_{\ZZ}(\Delta) \cdot \Eu(\Delta).
\end{equation}

Similarly, by taking a generic linear subspace
\begin{equation}
H_{i+1}=\left\{[\xi_1:\xi_2:\cdots :\xi_{m+1}] \in \PP^m \ \left|\ \dsum_{j=1}^{m+1}a_{j}^{(k)}\xi_j=0 \ \ (k=1,2,\ldots, i+1) \right. \right\}
\end{equation}
($a_{j}^{(k)}\in \CC$) of $\PP^m$ of codimension $i+1$ and using Theorem \ref{thm:2-10}, we have
\begin{equation}\label{eq:2-4-4}
\int_{H_{i+1}} \Eu_{X_A} =\dsum_{\begin{subarray}{c} \Delta \prec P\\ \d \Delta \geq i+1 \end{subarray}} (-1)^{\d \Delta -i-1} \binom{\d\Delta -1}{i} \Vol_{\ZZ}(\Delta)\cdot \Eu(\Delta).
\end{equation}
By \eqref{eq:2-4-1}, \eqref{eq:2-4-2}, \eqref{eq:2-4-3} and \eqref{eq:2-4-4}, we finally obtain
\begin{equation}
\delta_i :=\sum_{\Delta \prec P}(-1)^{\cd \Delta} \left\{\binom{\d \Delta-1}{i} +(-1)^{i-1}(i+1)\right\}
\Vol_{\ZZ}(\Delta) \cdot \Eu(\Delta).
\end{equation}
This completes the proof.
\end{proof}

\begin{corollary}
Assume that $X_A^*$ is a hypersurface in $(\PP^m)^*$. Then the degree of the $A$-discriminant is given by
\begin{equation}
\deg X_A^*=\sum_{\Delta \prec P}(-1)^{\cd \Delta}(\d \Delta +1) \Vol_{\ZZ}(\Delta) \cdot \Eu(\Delta).
\end{equation}
\end{corollary}

Note that if the dual defect of $X_A$ is zero the degree formula of $X_A^*$ for singular $X_A$'s was also obtained by Dickenstein-Feichtner-Sturmfels \cite{D-F-S}. In their paper, they express the degree of $X_A^*$ by other combinatorial invariants of $A$. However our formulas seem to be more directly related to the geometry of the convex polytope $P$. For example, if $X_A$ is smooth, our formula coincides with Gelfand-Kapranov-Zelevinsky's theorem \cite[Chapter 9, Theorem 2.8]{G-K-Z}.

In Section \ref{sec:4}, we will give two combinatorial formulas for the Euler obstruction $\Eu_{X_A} \colon X_A \longrightarrow \ZZ$ of $X_A$. Together with Theorem \ref{thm:2-4} above, we can calculate the dimension and the degree of $X_A^*$ for any $A \subset M=\ZZ^n$ (see Example \ref{exa:6-8}).

\section{Euler obstructions of toric varieties}\label{sec:4}

In this section, we give some formulas for the Euler obstructions of toric varieties. A beautiful formula for the Euler obstructions of $2$-dimensional normal toric varieties was proved by Gonzalez-Sprinberg \cite{G-S}. Our result can be considered as a natural generalization of his formula.

First we recall the definition of Euler obstructions (for the detail see \cite{Kashiwara} etc.). Let $X$ be an algebraic variety over $\CC$. Then the Euler obstruction $\Eu_X$ of $X$ is a $\ZZ$-valued constructible function on $X$ defined as follows. The value of $\Eu_X$ on the smooth part of $X$ is defined to be $1$. In order to define the value of $\Eu_X$ at a singular point $p \in X$, we take an affine open neighborhood $U$ of $p$ in $X$ and a closed embedding $U \longhookrightarrow \CC^m$. Next we choose a Whitney stratification $U=\bigsqcup_{\alpha \in A}U_{\alpha}$ of $U$ such that $U_{\alpha}$ are connected. Then the values $\Eu_X(U_{\alpha})$ of $\Eu_X$ on the strata $U_{\alpha}$ are defined by induction on codimensions of $U_{\alpha}$ as follows.
\begin{enumerate}\renewcommand{\labelenumi}{{\rm (\roman{enumi})}}
\item If $U_{\alpha}$ is contained in the smooth part of $U$, we set $\Eu_X(U_{\alpha})=1$.
\item Assume that for $k \geq 0$ the values of $\Eu_X$ on the strata $U_{\alpha}$ such that $\cd U_{\alpha} \leq k$ are already determined. Then for a stratum $U_{\beta}$ such that $\cd U_{\beta}=k+1$ the value $\Eu_X(U_{\beta})$ is defined by
\begin{equation}
\Eu_X(U_{\beta})=\sum_{U_{\beta} \subsetneq \overline{U_{\alpha}}} \chi(U_{\alpha} \cap f^{-1}(\eta) \cap B(q;\e)) \cdot \Eu_X(U_{\alpha})
\end{equation}
for sufficiently small $\e>0$ and $0 < \eta \ll \e$, where $q \in U_{\beta}$ and $f$ is a holomorphic function defined on an open neighborhood $W$ of $q$ in $\CC^m$ such that $U_{\beta} \cap W \subset f^{-1}(0)$ and $(q;{\rm grad} f(q)) \in T_{U_{\beta}}^*\CC^m \setminus \left(\bigcup_{U_{\beta} \subsetneq \overline{U_{\alpha}}} \overline{T_{U_{\alpha}}^*\CC^m}\right)$.
\end{enumerate}
The above integers $\chi(U_{\alpha} \cap f^{-1}(\eta) \cap B(q;\e))$ can be calculated by the nearby cycle functor $\psi_f$ as
\begin{equation} 
\chi(U_{\alpha} \cap f^{-1}(\eta) \cap B(q;\e)) =\chi (\psi_f(\CC_{U_{\alpha} } )_q). 
\end{equation} 
Indeed by \cite[Proposition 4.2.2]{Dimca}, for the Milnor fiber $F_q: =U \cap f^{-1}(\eta) \cap B(q;\e)=(f|_{U})^{-1}(\eta) \cap B(q;\e)$ of $f|_{U}\colon U \longrightarrow \CC$ at $q \in U_{\beta}$ we have an isomorphism 
\begin{equation} 
\RG (F_q ; \CC_{U_{\alpha} }) \simeq \psi_f(\CC_{U_{\alpha} } )_q. 
\end{equation} 
Since $F_q =\bigsqcup_{\gamma \in A} (U_{\gamma} \cap F_q)$ is also a Whitney stratification of $F_q$ ($f|_{U} \colon U \longrightarrow \CC$ has the isolated stratified critical value $0 \in \CC$ by \cite[Proposition 1.3]{Massey}), we obtain $\chi (\RG (F_q ; \CC_{U_{\alpha} }))=  \chi(U_{\alpha} \cap F_q)= \chi(U_{\alpha} \cap f^{-1}(\eta) \cap B(q;\e))$ by \cite[Theorem 4.1.22]{Dimca}.

\subsection{The case of affine toric varieties}\label{sec:4-1}

From now on, we shall consider the toric case. Let $N \simeq \ZZ^n$ be a $\ZZ$-lattice of rank $n$ and $\sigma$ a strongly convex rational polyhedral cone in $N_{\RR}=\RR \otimes_{\ZZ}N$. We denote by $M$ the dual lattice of $N$ and define the polar cone $\sigma^{\vee}$ of $\sigma$ in $M_{\RR}=\RR \otimes_{\ZZ}M$ by 
\begin{equation} 
\sigma^{\vee}=\{ v \in M_{\RR} \ | \ \langle u, v \rangle \geq 0 \quad \text{for any} \quad u \in \sigma \}.
\end{equation} 
Then the dimension of $\sigma^{\vee}$ is $n$ and we obtain a semigroup $\SS_{\sigma}:=\sigma^{\vee} \cap M$ and an $n$-dimensional affine toric variety $X:=U_{\sigma}=\Spec(\CC[\SS_{\sigma}])$ (see \cite{Fulton} and \cite{Oda} etc.). Recall also that the algebraic torus $T=\Spec (\CC [M])\simeq (\CC^*)^n$ acts naturally on $X=U_{\sigma}$ and the $T$-orbits in $X$ are indexed by the faces $\Delta_{\alpha} \prec \sigma^{\vee}$ of $ \sigma^{\vee}$. We denote by $\LL(\Delta_{\alpha})$ the smallest linear subspace of $M_{\RR}$ containing $\Delta_{\alpha}$. For a face $\Delta_{\alpha}$ of $\sigma^{\vee}$, denote by $T_{\alpha}$ the $T$-orbit $\Spec (\CC [M \cap \LL(\Delta_{\alpha})])$ which corresponds to $\Delta_{\alpha}$. Then we obtain a decomposition $X=\bigsqcup_{\Delta_{\alpha} \prec \sigma^{\vee}} T_{\alpha}$ of $X=U_{\sigma}$ into $T$-orbits. By the above recursive definition (ii) of $\Eu_X$, in order to compute the Euler obstruction $\Eu_X \colon X \longrightarrow \ZZ$ it suffices to determine the following numbers.

\begin{definition}\label{dfn:4-1}
For two faces $\Delta_{\alpha}$, $\Delta_{\beta}$ of $\sigma^{\vee}$ such that $\Delta_{\beta} \precneqq \Delta_{\alpha}$ ($\Longleftrightarrow$ $T_{\beta} \subsetneq \overline{T_{\alpha}}$), we define the linking number $l_{\alpha,\beta} \in \ZZ$ of $T_{\alpha}$ along $T_{\beta}$ as follows. For a point $q\in T_{\beta}$ and a closed embedding $\iota \colon X=U_{\sigma} \longhookrightarrow \CC^m$, we set
\begin{equation}
l_{\alpha,\beta}:=\chi (\psi_f (\CC_{T_{\alpha} })_q), 
\end{equation}
where $f$ is a holomorphic function defined on an open neighborhood $W$ of $q$ in $\CC^m$ such that $T_{\beta} \cap W \subset f^{-1}(0)$ and $(q; \grad f(q)) \in T_{T_{\beta}}^*\CC^m \setminus \left(\bigcup_{\Delta_{\beta} \precneqq \Delta_{\gamma}} \overline{T_{T_{\gamma}}^*\CC^m}\right)$.
\end{definition}

Note that the above definition of $l_{\alpha,\beta}$ does not depend on the choice of $q \in T_{\beta}$, $\iota$ and $f$ etc. We will show that $l_{\alpha,\beta}$ can be described by the geometry of the cones $\Delta_{\alpha}$ and $\Delta_{\beta}$. First let us consider the $\ZZ$-lattice $M_{\beta}:=M \cap \LL(\Delta_{\beta})$ of rank $\d \Delta_{\beta}$. Next set $\LL(\Delta_{\beta})^{\prime}:=M_{\RR}/\LL(\Delta_{\beta})$ and let $p_{\beta} \colon M_{\RR} \longrightarrow \LL(\Delta_{\beta})^{\prime}$ be the natural projection. Then $M_{\beta}^{\prime}:=p_{\beta}(M) \subset \LL(\Delta_{\beta})^{\prime}$ is a $\ZZ$-lattice of rank $(n-\d \Delta_{\beta})$ and $K_{\alpha, \beta}:=p_{\beta}(\Delta_{\alpha}) \subset \LL(\Delta_{\beta})^{\prime}$ is a proper convex cone with apex $0 \in \LL(\Delta_{\beta})^{\prime}$.

\begin{definition}\label{dfn:4-2}
For two faces $\Delta_{\alpha}$ and $\Delta_{\beta}$ of $\sigma^{\vee}$ such that $\Delta_{\beta} \precneqq \Delta_{\alpha}$, we define the normalized relative subdiagram volume $\RSV_{\ZZ}(\Delta_{\alpha}, \Delta_{\beta})$ of $\Delta_{\alpha}$ along $\Delta_{\beta}$ by
\begin{equation}
\RSV_{\ZZ}(\Delta_{\alpha}, \Delta_{\beta}):=\Vol_{\ZZ}(K_{\alpha,\beta} \setminus \Theta_{\alpha, \beta}),
\end{equation}
where $\Theta_{\alpha,\beta}$ is the convex hull of $K_{\alpha,\beta} \cap (M_{\beta}^{\prime} \setminus \{0\})$ in $\LL(\Delta_{\beta})^{\prime} \simeq \RR^{n- \d \Delta_{\beta}}$ and $\Vol_{\ZZ}(K_{\alpha,\beta} \setminus \Theta_{\alpha, \beta})$ is the normalized $(\d \Delta_{\alpha} -\d \Delta_{\beta})$-dimensional volume of $K_{\alpha,\beta} \setminus \Theta_{\alpha, \beta}$ with respect to the lattice $M_{\beta}^{\prime} \cap \LL(K_{\alpha,\beta})$. If $\Delta_{\alpha}=\Delta_{\beta}$, we set $\RSV_{\ZZ}(\Delta_{\alpha},\Delta_{\alpha}):=1$.
\end{definition}

\begin{theorem}\label{thm:4-3}
For two faces $\Delta_{\alpha}$ and $\Delta_{\beta}$ of $\sigma^{\vee}$ such that $\Delta_{\beta} \precneqq \Delta_{\alpha}$, the linking number $l_{\alpha, \beta}$ of $T_{\alpha}$ along $T_{\beta}$ is given by
\begin{equation}
l_{\alpha, \beta}=(-1)^{\d \Delta_{\alpha}-\d \Delta_{\beta}-1}\RSV_{\ZZ}(\Delta_{\alpha}, \Delta_{\beta}).
\end{equation}
\end{theorem}

\begin{proof}
First recall that we have $T_{\beta}=\Spec(\CC[M_{\beta}])\simeq (\CC^*)^{\d \Delta_{\beta}}$. For each face $\Delta_{\alpha}$ of $\sigma^{\vee}$ such that $\Delta_{\beta} \precneqq \Delta_{\alpha}$, consider the semigroups $\SS_{\alpha}:=M \cap \Delta_{\alpha}$ and $\SS_{\alpha,\beta}:=M_{\beta}^{\prime} \cap K_{\alpha,\beta}$. In the special case when $\Delta_{\alpha}=\sigma^{\vee}$, we set also $\SS_{\sigma, \beta}:=M_{\beta}^{\prime}\cap p_{\beta}(\sigma^{\vee})$. Then for any face $\Delta_{\alpha} \prec \sigma^{\vee}$ such that $\Delta_{\beta} \precneqq \Delta_{\alpha}$ it is easy to see that
\begin{equation}
\SS_{\alpha}+M_{\beta}=\SS_{\alpha, \beta}\oplus M_{\beta}
\end{equation}
and in a neighborhood of $T_{\beta}$ in $X$ we have
\begin{eqnarray}
\overline{T_{\alpha}}
&=&\Spec(\CC[\SS_{\alpha}+M_{\beta}])\\
&=&\Spec(\CC[\SS_{\alpha, \beta}]) \times T_{\beta}
\end{eqnarray}
(see the proof of \cite[Chapter 5, Theorem 3.1]{G-K-Z}). In particular, for $\Delta_{\alpha}=\sigma^{\vee}$ we have
\begin{equation}
X=\Spec(\CC[\SS_{\sigma, \beta}]) \times T_{\beta}
\end{equation}
in a neighborhood of $T_{\beta}$. More precisely, there exists a unique point $q \in X_{\sigma, \beta}:=\Spec(\CC[\SS_{\sigma, \beta}])$ such that $\{ q\} \times T_{\beta} = T_{\beta}$. Now let us take a face $\Delta_{\alpha} \prec \sigma^{\vee}$ such that $\Delta_{\beta} \precneqq \Delta_{\alpha}$ and set $X_{\alpha, \beta}:=\Spec(\CC[\SS_{\alpha, \beta}])$. Then by the inclusion $\SS_{\alpha, \beta}\longhookrightarrow \SS_{\sigma, \beta}$ we obtain a surjective homomorphism
\begin{equation}
\CC[\SS_{\sigma, \beta}] \longtwoheadrightarrow \CC[\SS_{\alpha, \beta}]
\end{equation}
of $\CC$-algebras and hence a closed embedding $X_{\alpha, \beta} \longhookrightarrow X_{\sigma, \beta}$. Denote by $T_{\alpha, \beta}$ the open dense torus $\Spec(\CC[M_{\beta}^{\prime} \cap \LL(K_{\alpha,\beta})])\simeq (\CC^*)^{\d \Delta_{\alpha}-\d \Delta_{\beta}}$ of the toric variety $X_{\alpha, \beta}$. Note that we have $T_{\alpha} \simeq T_{\alpha, \beta} \times T_{\beta}$. Now let $v_1, v_2, \ldots , v_m $ be generators of the semigroup $\SS_{\sigma, \beta}$ and consider a surjective morphism
\begin{equation}
\CC [t_1, t_2, \ldots , t_m] \longrightarrow \CC [\SS_{\sigma, \beta}]
\end{equation}
of $\CC$-algebras defined by $t_i \longmapsto [v_i]$. Then it induces a closed embedding $X_{\sigma, \beta} \longhookrightarrow \CC^m$ by which the point $q \in X_{\sigma, \beta}$ is sent to $0 \in \CC^m$. If we consider $T_{\alpha, \beta}$ as a locally closed subset of $\CC^m$ by this embedding, then the linking number $l_{\alpha, \beta}$ of $T_{\alpha}$ along $T_{\beta}$ is given by
\begin{equation}
l_{\alpha, \beta}=\chi(\psi_f(\CC_{T_{\alpha, \beta} })_0),
\end{equation}
where $f \colon \CC^m \longrightarrow \CC$ is a generic linear form. By applying Theorem \ref{thm:pre-10} to the closed embedding $X_{\alpha, \beta}\longhookrightarrow \CC^m$, we obtain
\begin{equation}
l_{\alpha, \beta}=\chi(\psi_g(\CC_{T_{\alpha, \beta}})_0),
\end{equation}
where we set $g:=f|_{X_{\alpha, \beta}}$. Finally it follows from \cite[Corollary 3.6]{M-T-new2} (whose special case used here can be deduced also from the proof of \cite[Chapter 10, Theorem 2.12]{G-K-Z}) that
\begin{equation}
l_{\alpha,\beta}=(-1)^{\d \Delta_{\alpha}-\d \Delta_{\beta}-1} \Vol_{\ZZ}(K_{\alpha,\beta} \setminus \Theta_{\alpha, \beta}).
\end{equation}
This completes the proof.
\end{proof}

Since the Euler obstruction $\Eu_X \colon X \longrightarrow \ZZ$ of $X$ is constant on each $T$-orbit $T_{\alpha}$ ($\Delta_{\alpha} \prec \sigma^{\vee}$), we denote by $\Eu(\Delta_{\alpha})$ the value of $\Eu_X$ on $T_{\alpha}$. Then we have

\begin{corollary}\label{cor:6-3-1}
All the values $\Eu(\Delta_{\alpha})$ of $\Eu_X \colon X \longrightarrow \ZZ$ are determined by induction on codimensions of faces of $\sigma^{\vee}$ as follows:
\begin{enumerate}\renewcommand{\labelenumi}{{\rm (\roman{enumi})}}
\item $\Eu(\sigma^{\vee}):=\Eu_X(T)=1$,
\item $\Eu(\Delta_{\beta})=\sum_{\Delta_{\beta} \precneqq \Delta_{\alpha}} (-1)^{\d \Delta_{\alpha} -\d \Delta_{\beta} -1} \RSV_{\ZZ}(\Delta_{\alpha}, \Delta_{\beta}) \cdot \Eu(\Delta_{\alpha})$.
\end{enumerate}
\end{corollary}

\subsection{The case of toric varieties associated with lattice points I}\label{sec:4-2}

Now let us consider a special case of projective toric varieties associated with lattice points. We inherit the situation and the notations in Section \ref{sec:2}. Let $A \subset M=\ZZ^n$ be a finite subset of $M=\ZZ^n$ such that the convex hull $P$ of $A$ in $M_{\RR}$ is $n$-dimensional. Let $N=\Hom_{\ZZ}(M,\ZZ)=M^*$ be the dual $\ZZ$-lattice of $M$ and set $N_{\RR}:=\RR \otimes_{\ZZ}N$. Since $N_{\RR}$ is the dual vector space of $M_{\RR}$, a point $u \in N_{\RR}$ is a linear form on $M_{\RR}$.

\begin{definition}
For $u \in N_{\RR}$ we define the supporting face $\Delta(P,u)$ of $u$ in $P$ by
\begin{equation}
\Delta(P,u):=\left\{ v \in P \ \left|\ \langle u,v\rangle =\min_{w \in P}\langle u,w \rangle \right. \right\}.
\end{equation}
\end{definition}

For each face $\Delta \prec P$ of $P$, set
\begin{equation}
\sigma_{\Delta}:=\{ u \in N_{\RR} \ |\ \Delta(P,u) =\Delta\}.
\end{equation}
Then we obtain a decomposition of $N_{\RR}$:
\begin{equation}
N_{\RR}=\bigsqcup_{\Delta \prec P} \sigma_{\Delta}
\end{equation}
and $\Sigma_P:=\{ \overline{\sigma_{\Delta}}\ |\ \Delta \prec P\}$ is a complete fan in $N_{\RR}$. We call $\Sigma_P$ the normal fan of $P$. Let $X_{\Sigma_P}$ be the complete toric variety associated with $\Sigma_P$ and denote by $T$ its open dense torus. By \cite[Theorem 2.13]{Oda}, if $P$ is sufficiently large and $A=P \cap M$, the natural morphism $\Phi_A \colon X_{\Sigma_P} \longrightarrow \PP^{\sharp A-1}$ associated with $A$ induces an isomorphism $X_{\Sigma_P} \simto X_A$. Note that in this case the toric variety $X_A$ is normal. Let us give a formula for $\Eu_{X_A}$ in this special but important case where $X_{\Sigma_P} \simto X_A$. For a face $\Delta_{\alpha} \prec P$ of $P$, we denote by $T_{\alpha}$ the $T$-orbit in $X_{\Sigma_P} \simeq X_A$ which corresponds to the cone $\overline{ \sigma_{\Delta_{\alpha}}}$ in $\Sigma_P$. Then we obtain a decomposition $X_A\simeq X_{\Sigma_P}=\bigsqcup_{\Delta_{\alpha} \prec P} T_{\alpha}$ of $X_A$ into $T$-orbits. Now let $\Delta_{\alpha}$, $\Delta_{\beta}$ be two faces of $P$ such that $\Delta_{\beta}\precneqq \Delta_{\alpha}$. Let $v \in \Delta_{\beta}$ be a vertex of of the smaller face $\Delta_{\beta}$ and $\sigma := \overline{\sigma_{ \{ v\} }} \in \Sigma_P$ the maximal cone which corresponds to the $0$-dimensional face $\{v \} \prec P$. Then $T_{\alpha}$ and $T_{\beta}$ are contained in the affine open subset $U_{\sigma}= \Spec (\CC [\sigma^{\vee} \cap M])\subset X_{\Sigma_P}$ and we can apply our results in Subsection \ref{sec:4-1}. Indeed, by the dilation action of the multiplicative group $\RR_{>0}$ on $M_{\RR}$, we obtain the equality $\RR_{>0}(P-v)=\sigma^{\vee}$ in $M_{\RR}$ and the natural correspondence:
\begin{equation}
\{ \text{faces of $P$ containing $v$}\} \overset{\text{1:1}}{\longleftrightarrow} \{ \text{faces of $\sigma^{\vee}$}\}.
\end{equation}
Note that this correspondence is compatible with the ones for $T$-orbits in $X_A$ and $X_{\Sigma_P}$. Therefore, by taking the two faces of $\sigma^{\vee}$ which correspond to $\Delta_{\alpha}$ and $\Delta_{\beta}$ via this correspondence, we can define the normalized relative subdiagram volume $\RSV_{\ZZ}(\Delta_{\alpha}, \Delta_{\beta})$ by Definition \ref{dfn:4-2} and give a formula for the Euler obstruction $\Eu_{X_A} \colon X_A \longrightarrow \ZZ$ as follows. Since $\Eu_{X_A}$ is constant on each $T$-orbit, for a face $\Delta_{\alpha} \prec P$ of $P$ denote by $\Eu(\Delta_{\alpha})$ the value of $\Eu_{X_A}$ on $T_{\alpha}$. Then by Corollary \ref{cor:6-3-1} all the values $\Eu(\Delta_{\alpha})$ are determined by induction on codimensions of faces of $P$ as follows:
\begin{enumerate}\renewcommand{\labelenumi}{{\rm (\roman{enumi})}}
\item $\Eu(P):=\Eu_{X_A}(T)=1$,
\item $\Eu(\Delta_{\beta})=\sum_{\Delta_{\beta} \precneqq\Delta_{\alpha}} (-1)^{\d \Delta_{\alpha} -\d \Delta_{\beta} -1} \RSV_{\ZZ}(\Delta_{\alpha}, \Delta_{\beta}) \cdot \Eu(\Delta_{\alpha})$.
\end{enumerate}

\begin{example}\label{exa:6-8}
We give an example of integral convex polytopes for which the degree of the $A$-discriminant is easily computed by our method. For a $\ZZ$-basis $\{m_1,m_2,m_3\}$ of $M \simeq \ZZ^3$, let $P$ be the $3$-dimensional simplex with vertices $v_1=m_1$, $v_2=m_2$, $v_3=2m_3$, $v_4=0$ and set $A:=P \cap M=\{0, m_1, m_2, m_3, 2m_3\}$. Then we can easily check that the condition in \cite[Theorem 2.13]{Oda} is satisfied. Namely the line bundle on $X_{\Sigma_P}$ associated with $P$ is very ample and $X_A \simeq X_{\Sigma_P}$ in $\PP^4$ in this case.

Let us compute the values of the Euler obstruction $\Eu_{X_A}$ of $X_A$ by our algorithm.

For $\alpha \subset \{1,2,3,4\}$, we denote by $\Delta_{\alpha}$ the face of $P$ whose vertices are $\{v_i \ |\ i \in \alpha\}$.

We can easily determine the values of $\Eu_{X_A}$ on the $2$ and $3$-dimensional $T$-orbits:
\begin{equation}\label{eq:5-1}
\Eu(P)=\Eu(\Delta_{123})=\Eu(\Delta_{124})=\Eu(\Delta_{134})=\Eu(\Delta_{234})=1.
\end{equation}

Starting from the values \eqref{eq:5-1}, we can determine the values of the Euler obstruction $\Eu_{X_A}$ on $1$-dimensional $T$-orbits:
\begin{equation}\label{eq:5-2}
\Eu(\Delta_{12})=0, \ \ \Eu(\Delta_{13})=\Eu(\Delta_{14})=\Eu(\Delta_{23})=\Eu(\Delta_{24})=\Eu(\Delta_{34})=1.
\end{equation}
For example, $\Eu(\Delta_{12})$ is computed as follows. 
\begin{eqnarray}
\Eu(\Delta_{12})
&=&-\RSV_{\ZZ}(P,\Delta_{12})\Eu(P)+\RSV_{\ZZ}(\Delta_{123},\Delta_{12})\Eu(\Delta_{123}) \\
& &+\RSV_{\ZZ}(\Delta_{124},\Delta_{12})\Eu(\Delta_{124}) \nonumber \\
&=&-2\cdot 1+1\cdot 1 +1\cdot 1=0.
\end{eqnarray}
Moreover, the values of the Euler obstruction $\Eu_{X_A}$ on $0$-dimensional $T$-orbits are determined from the values \eqref{eq:5-1} and \eqref{eq:5-2}:
\begin{equation}
\Eu(\Delta_1)=\Eu(\Delta_2)=0, \ \ \Eu(\Delta_3)=\Eu(\Delta_4)=1.
\end{equation}
Now let us compute the codimension and degree of the dual variety $X_A^*$ of $X_A$. By \eqref{eq:2-9}, we have
\begin{equation}
\delta_1=\delta_2=0, \hspace{10mm}\delta_3=2.
\end{equation}
Then by Theorem \ref{thm:2-4} we obtain
\begin{equation}\label{eq:4-27}
\cd X_A^*=3, \ \ \deg X_A^*=2.
\end{equation}

In this case the map $\Phi_A \colon (\CC^*)^3 \longrightarrow \PP^4$ is defined by $(a,b,c)\longmapsto [a:b:c:c^2:1]$. Then in the affine chart $\{ [\xi_1:\xi_2:\xi_3:\xi_4:\xi_5] \in \PP^4\ |\ \xi_1\neq 0\} \simeq \CC^4$ we have $X_A=\overline{{\rm im}\Phi_A} = \{(x,y,z,w)\in \CC^4\ |\ wz=y^2\}$. By this expression we can show \eqref{eq:4-27} directly. For a list of $X_A$ with large dual defect, see the recent results in \cite{C-C} and \cite{DiRocco}.
\end{example}

\subsection{The case of toric varieties associated with lattice points II}
From now on, we shall give a combinatorial description of $\Eu_{X_A}$ for the varieties $X_A$ defined by general finite subsets $A \subset M \simeq \ZZ^n$. We inherit the notations in Section \ref{sec:2}. Without loss of generality, we may assume that the rank of the affine $\ZZ$-lattice $M(A)$ generated by $A$ is $n$. Let $P$ be the convex hull of $A$ in $M_{\RR}$. For each face $\Delta_{\alpha}$ of $P$, consider the smallest affine subspace $\LL(\Delta_{\alpha})$ of $M_{\RR}$ containing $\Delta_{\alpha}$ and the affine $\ZZ$-lattice $M_{\alpha}:= M(A \cap \Delta_{\alpha})$ generated by $A \cap \Delta_{\alpha}$ in $\LL(\Delta_{\alpha})$. Now let us fix two faces $\Delta_{\alpha}$, $\Delta_{\beta}$ of $P$ such that $\Delta_{\beta} \prec \Delta_{\alpha}$. By taking a suitable affine transformation of the lattice $M(A)$, we may assume that the origin $0$ of $M(A)$ is a vertex of the smaller face $\Delta_{\beta}$. By this choice of the origin $0 \in \Delta_{\beta} \cap M(A)$, we define the subsemigroup $\SS_{\alpha}$ of $M_{\alpha}$ generated by $A \cap \Delta_{\alpha}$. Although $\SS_{\alpha}$ depends also on $\Delta_{\beta}$ etc., we denote it by $\SS_{\alpha}$ to simplify the notation. Denote by $M_{\alpha}/\Delta_{\beta}$ the quotient lattice $M_{\alpha}/ (M_{\alpha} \cap \LL(\Delta_{\beta}))$ of rank $(\d \Delta_{\alpha}-\d \Delta_{\beta})$. Then the following definitions are essentially due to \cite[Chapter 5, page 178]{G-K-Z}.

\begin{definition}[\cite{G-K-Z}]\label{dfn:6-5}
\begin{enumerate}\renewcommand{\labelenumi}{{\rm (\roman{enumi})}}
\item We denote by $\SS_{\alpha}/\Delta_{\beta}$ the image of $\SS_{\alpha} \subset M_{\alpha}$ in the quotient $\ZZ$-lattice $M_{\alpha}/\Delta_{\beta}$.
\item We denote by $K(\SS_{\alpha}/\Delta_{\beta})$ (resp. $K_+(\SS_{\alpha}/\Delta_{\beta})$) the convex hull of $\SS_{\alpha}/\Delta_{\beta}$ (resp. $(\SS_{\alpha}/\Delta_{\beta})\setminus \{0\}$) in $(M_{\alpha}/\Delta_{\beta})_{\RR}$ and set 
\begin{equation}
K_-(\SS_{\alpha}/\Delta_{\beta}):=\overline{K(\SS_{\alpha}/\Delta_{\beta}) \setminus K_+(\SS_{\alpha}/\Delta_{\beta})}.
\end{equation}
We call $K_-(\SS_{\alpha}/\Delta_{\beta})$ the subdiagram part of the semigroup $\SS_{\alpha}/\Delta_{\beta}$ and denote by $u(\SS_{\alpha}/\Delta_{\beta})$ its normalized $(\d \Delta_{\alpha}-\d \Delta_{\beta})$-dimensional volume with respect to the $\ZZ$-lattice $M_{\alpha}/\Delta_{\beta} \subset (M_{\alpha}/\Delta_{\beta})_{\RR}$. If $\Delta_{\alpha}=\Delta_{\beta}$, we set $u(\SS_{\alpha}/\Delta_{\alpha}):=1$.
\end{enumerate}
\end{definition}

Finally, recall the definition of the index $i(\Delta_{\alpha}, \Delta_{\beta}) \in \ZZ_{>0}$ in \cite[Chapter 5, (3.1)]{G-K-Z}.

\begin{definition}[\cite{G-K-Z}]\label{dfn:6-6}
For two faces $\Delta_{\alpha}$, $\Delta_{\beta}$ of $P$ such that $\Delta_{\beta} \prec \Delta_{\alpha}$, we define $i(\Delta_{\alpha}, \Delta_{\beta})$ to be the index
\begin{equation}
i(\Delta_{\alpha}, \Delta_{\beta}):=[M_{\alpha} \cap \LL(\Delta_{\beta}):M_{\beta}]
\end{equation}
\end{definition}

Now recall that by \cite[Chapter 5, Proposition 1.9]{G-K-Z} we have the basic correspondence:
\begin{equation}
\{ \text{faces of $P$}\} \overset{\text{1:1}}{\longleftrightarrow} \{ \text{$T$-orbits in $X_A$}\}.
\end{equation}
For a face $\Delta_{\alpha} \prec P$ of $P$, we denote by $T_{\alpha}$ the corresponding $T$-orbit in $X_A$. We also denote by $\Eu(\Delta_{\alpha})$ the value of the Euler obstruction $\Eu_{X_A} \colon X_A \longrightarrow \ZZ$ on $T_{\alpha}$.

\begin{theorem}\label{thm:6-7}
The values $\Eu(\Delta_{\alpha})$ are determined by:
\begin{enumerate}\renewcommand{\labelenumi}{{\rm (\roman{enumi})}}
\item $\Eu(P)=1$,
\item $\Eu(\Delta_{\beta})=\sum_{\Delta_{\beta} \precneqq \Delta_{\alpha}} (-1)^{\d \Delta_{\alpha}-\d \Delta_{\beta}-1}i(\Delta_{\alpha}, \Delta_{\beta}) \cdot u(\SS_{\alpha}/\Delta_{\beta}) \cdot \Eu(\Delta_{\alpha})$.
\end{enumerate}
\end{theorem}

\begin{proof}
Let $\Delta_{\alpha} \prec P$ be a face of $P$. Then by \cite[Chapter 5, Proposition 1.9]{G-K-Z} the closure $\overline{T_{\alpha}}$ of $T_{\alpha}$ in $X_A$ is isomorphic to the projective toric variety $X_{A \cap \Delta_{\alpha}} \subset \PP^{\sharp (A \cap \Delta_{\alpha})-1}$ defined by the finite subset $A \cap \Delta_{\alpha}$ in the lattice $M_{\alpha}=M(A \cap \Delta_{\alpha})\simeq \ZZ^{\d \Delta_{\alpha}}$. Moreover the cone ${\rm Cone}(\overline{T_{\alpha}})\subset \CC^{\sharp (A \cap \Delta_{\alpha})}$ over $\overline{T_{\alpha}}\subset \PP^{\sharp (A \cap \Delta_{\alpha})-1}$ is an affine variety as follows. Let
\begin{equation}
i_{\alpha} \colon M_{\alpha} \longhookrightarrow \Xi_{\alpha} :=M_{\alpha} \oplus \ZZ\simeq \ZZ^{\d \Delta_{\alpha}+1} 
\end{equation}
be the embedding defined by $v\longmapsto (v,1)$ and $\tl{\SS}_{\alpha}$ the subsemigroup of the lattice $\Xi_{\alpha}$ generated by $i_{\alpha}(A \cap \Delta_{\alpha})$ and $0 \in \Xi_{\alpha}$. Then by \cite[Chapter 5, Proposition 2.3]{G-K-Z} the cone ${\rm Cone}(\overline{T_{\alpha}}) \subset \CC^{\sharp (A \cap \Delta_{\alpha})}$ is isomorphic to the affine toric variety $\Spec(\CC[\tl{\SS}_{\alpha}])$. In the special case when $\Delta_{\alpha}=P$, we set $\Xi:=\Xi_{\alpha}$ ($=M(A)\oplus \ZZ$) and $\tl{\SS} :=\tl{\SS}_{\alpha}$ for short. Since $\tl{\SS}_{\alpha}$ is a subsemigroup of $\tl{\SS}$ via the inclusions $M_{\alpha} \subset M(A)$ and $\Xi_{\alpha} \subset \Xi$, there exists a natural surjection
\begin{equation}
\CC[\tl{\SS}] \longtwoheadrightarrow \CC[\tl{\SS}_{\alpha}].
\end{equation}
This corresponds to the closed embedding
\begin{equation}
{\rm Cone}(\overline{T_{\alpha}}) \simeq \Spec(\CC[\tl{\SS}_{\alpha}])\longhookrightarrow {\rm Cone}(X_A) \simeq \Spec(\CC[\tl{\SS}]).
\end{equation}
Now let $\Delta_{\alpha}$ and $\Delta_{\beta}$ be two faces of $P$ such that $\Delta_{\beta} \precneqq \Delta_{\alpha}$ ($\Longleftrightarrow$ $T_{\beta} \subsetneq \overline{T_{\alpha}}$). We have to determine the linking number $l_{\alpha, \beta}$ of $T_{\alpha}$ along $T_{\beta}$ (defined as in Definition \ref{dfn:4-1}). Since the singularity of $\overline{T_{\alpha}}$ along $T_{\beta}$ is the same as that of ${\rm Cone}(\overline{T_{\alpha}})\simeq \Spec(\CC[\tl{\SS}_{\alpha}])$ along ${\rm Cone}(T_{\beta})\simeq \Spec (\CC[\Xi_{\beta}]) \simeq (\CC^*)^{\d \Delta_{\beta} +1}$, it suffices to study the pair ${\rm Cone}(T_{\beta}) \subset {\rm Cone}(\overline{T_{\alpha}})$. Moreover, by the proof of \cite[Chapter 5, Theorem 3.1]{G-K-Z}, in a neighborhood of ${\rm Cone}(T_{\beta})$ in ${\rm Cone}(X_A)\subset \CC^{\sharp A}$, we have
\begin{gather}
{\rm Cone}(\overline{T_{\alpha}}) = \Spec(\CC[\tl{\SS}_{\alpha}+\Xi_{\beta}]),\\
{\rm Cone}(X_A)= \Spec(\CC[\tl{\SS} +\Xi_{\beta}])
\end{gather}
and the fibers of the morphisms
\begin{gather}
{\rm Cone}(\overline{T_{\alpha}})\simeq \Spec(\CC[\tl{\SS}_{\alpha}+\Xi_{\beta}])\longrightarrow {\rm Cone}(T_{\beta})\simeq \Spec (\CC[\Xi_{\beta}]),\\
{\rm Cone}(X_A)\simeq \Spec(\CC[\tl{\SS} +\Xi_{\beta}]) \longrightarrow {\rm Cone}(T_{\beta})\simeq \Spec (\CC[\Xi_{\beta}])
\end{gather}
induced by $\Xi_{\beta} \subset \tl{\SS}_{\alpha}+\Xi_{\beta}$ and $\Xi_{\beta} \subset \tl{\SS}+\Xi_{\beta}$ are $\Spec(\CC[(\tl{\SS}_{\alpha}+\Xi_{\beta})/\Xi_{\beta}])$ and $\Spec(\CC[(\tl{\SS}+\Xi_{\beta})/\Xi_{\beta}])$ respectively. Let us set
\begin{gather}
Y_{\alpha}:=\Spec(\CC[(\tl{\SS}_{\alpha}+\Xi_{\beta})/\Xi_{\beta}]),\\
Y:=\Spec(\CC[(\tl{\SS}+\Xi_{\beta})/\Xi_{\beta}]).
\end{gather}
Since the natural morphism
\begin{equation}
(\tl{\SS}_{\alpha}+\Xi_{\beta})/\Xi_{\beta}\longrightarrow (\tl{\SS}+\Xi_{\beta})/\Xi_{\beta}
\end{equation}
is injective, we obtain a surjection
\begin{equation}
\CC[(\tl{\SS}+\Xi_{\beta})/\Xi_{\beta}]\longtwoheadrightarrow \CC[(\tl{\SS}_{\alpha}+\Xi_{\beta})/\Xi_{\beta}]
\end{equation}
and hence a closed embedding $Y_{\alpha} \longhookrightarrow Y$. Note that $Y \cap {\rm Cone}(T_{\beta})=Y_{\alpha} \cap {\rm Cone}(T_{\beta})$ consists of a single point. We denote this point by $q$. Now let us consider the open subset $W_{\alpha}:=\Spec(\CC[\Xi_{\alpha}/\Xi_{\beta}])$ of $Y_{\alpha}=\Spec(\CC[(\tl{\SS}_{\alpha}+\Xi_{\beta})/\Xi_{\beta}])$. It is easy to see that $W_{\alpha}$ is the intersection of ${\rm Cone}(T_{\alpha})\simeq (\CC^*)^{\d \Delta_{\alpha}+1}$ and $Y_{\alpha}$. Let $v_1, v_2,\ldots , v_m $ be generators of the semigroup $(\tl{\SS}+\Xi_{\beta})/\Xi_{\beta}$ and consider a surjective morphism
\begin{equation}
\CC [t_1, t_2, \ldots , t_m] \longrightarrow \CC [(\tl{\SS}+\Xi_{\beta})/\Xi_{\beta}]
\end{equation}
of $\CC$-algebras defined by $t_i \longmapsto [v_i]$. Then it induces a closed embedding $Y \longhookrightarrow \CC^m$ by which the point $q \in Y$ is sent to $0 \in \CC^m$. If we consider $W_{\alpha}$ as a locally closed subset of $\CC^m$ by this embedding, then the linking number $l_{\alpha, \beta}$ is given by
\begin{equation}
l_{\alpha, \beta}=\chi(\psi_f(\CC_{W_{\alpha}})_0),
\end{equation}
where $f \colon \CC^m \longrightarrow \CC$ is a generic linear form. By applying Theorem \ref{thm:pre-10} to the closed embedding $Y_{\alpha} \longhookrightarrow \CC^m$, we obtain also
\begin{equation}
l_{\alpha, \beta}=\chi(\psi_g(\CC_{W_{\alpha}})_0),
\end{equation}
where we set $g:=f|_{Y_{\alpha}}$. In order to calculate this last term $\chi(\psi_g(\CC_{W_{\alpha}})_0)$, we shall investigate the structure of $W_{\alpha}$ more precisely. By the inclusion $(\Xi_{\beta})_{\RR} \subset (\Xi_{\alpha})_{\RR}$, we set $\Xi_{\alpha}^{\prime}:=\Xi_{\alpha}\cap (\Xi_{\beta})_{\RR}$. Since we assumed that the origin of the lattice $M_{\alpha}$ is a vertex of $\Delta_{\beta}$, the two lattices $\Xi_{\alpha}^{\prime}$ and $\Xi_{\beta}$ contain the subgroup $\{ (0,t) \in \Xi_{\alpha}\ | \ t \in \ZZ\} \simeq \ZZ$ of $\Xi_{\alpha}$. Hence we obtain an isomorphism
\begin{equation}
\Xi_{\alpha}^{\prime}/\Xi_{\beta}\simeq (M_{\alpha} \cap \LL(\Delta_{\beta}))/M_{\beta}.
\end{equation}
Namely $\Xi_{\beta}$ is a sublattice of $\Xi_{\alpha}^{\prime}$ with index $l:=i(\Delta_{\alpha}, \Delta_{\beta})$. By the fundamental theorem of finitely generated abelian groups, we may assume that $G:=\Xi_{\alpha}^{\prime}/\Xi_{\beta}$ is a cyclic group $\ZZ/l \ZZ$ of order $l=i(\Delta_{\alpha}, \Delta_{\beta})$. Now let us take a sublattice $\Xi_{\alpha}^{\prime \prime}$ of $\Xi_{\alpha}$ such that $\Xi_{\alpha}=\Xi_{\alpha}^{\prime}\oplus \Xi_{\alpha}^{\prime \prime}$. Then we have
\begin{equation}
\Xi_{\alpha}/\Xi_{\beta}\simeq G \oplus \Xi_{\alpha}^{\prime \prime}
\end{equation}
and
\begin{eqnarray}
W_{\alpha}
&\simeq &\Spec(\CC [G]) \times \Spec (\CC[\Xi_{\alpha}^{\prime \prime}])\\
&\simeq &\{ z\in \CC \ |\ z^l=1\} \times (\CC^*)^{\d \Delta_{\alpha} -\d \Delta_{\beta}}.
\end{eqnarray}
Let $\Psi \colon \Xi_{\alpha} \longtwoheadrightarrow G=\ZZ/l\ZZ$ be the composite of
\begin{equation}
\Xi_{\alpha}\longtwoheadrightarrow \Xi_{\alpha}^{\prime}\longtwoheadrightarrow G=\Xi_{\alpha}^{\prime}/\Xi_{\beta}.
\end{equation}
For $s \in \Xi_{\alpha}$, we define an integer $e(s)\in \{0,1,2, \ldots, l-1\}$ by $\Psi(s)=[e(s)]\in G\simeq \ZZ/l\ZZ$. Then for $k=0,1,2, \ldots, l-1$ there exist surjective homomorphisms
\begin{equation}
I_k \colon \CC[(\tl{\SS}_{\alpha}+\Xi_{\beta})/\Xi_{\beta}] \longtwoheadrightarrow \CC[(\tl{\SS}_{\alpha}+\Xi_{\alpha}^{\prime})/\Xi_{\alpha}^{\prime}]
\end{equation}
of $\CC$-algebras defined by
\begin{equation}
\sum_{s_i \in \tl{\SS}_{\alpha}} a_i\cdot [s_i+\Xi_{\beta}] \longmapsto \sum_{s_i \in \tl{\SS}_{\alpha}}a_i \cdot \mu_l^{ke(s_i)}\cdot [s_i+\Xi_{\alpha}^{\prime}],
\end{equation}
where $\mu_l=\exp{\(\frac{2\pi\sqrt{-1}}{l}\)}$ is the primitive $l$-th root of unity. On the other hand, since $\{ (0,t) \in \Xi_{\alpha}\ | \ t \in \ZZ\} \simeq \ZZ$ is a subgroup of $\Xi_{\alpha}^{\prime}$, we have isomorphisms
\begin{gather}
\Xi_{\alpha}/\Xi_{\alpha}^{\prime}\simeq M_{\alpha}/\Delta_{\beta}=M_{\alpha}/(M_{\alpha}\cap \LL(\Delta_{\beta})),\\
(\tl{\SS}_{\alpha}+\Xi_{\alpha}^{\prime})/\Xi_{\alpha}^{\prime}\simeq \SS_{\alpha}/\Delta_{\beta}.
\end{gather}
Let us set $Z_{\alpha}:=\Spec(\CC[\SS_{\alpha}/\Delta_{\beta}])$. Then by the above surjective homomorphisms $I_k$ ($k=0,1,2,\ldots,l-1$) we obtain closed embeddings
\begin{equation}\label{eq:4-53}
\iota_k \colon Z_{\alpha} \longhookrightarrow Y_{\alpha} \quad (k=0,1,2,\ldots,l-1).
\end{equation}
We can see that the images of these embeddings $\iota_k \colon Z_{\alpha} \longhookrightarrow Y_{\alpha}$ are the explicit realizations of the branches along $T$-orbits found in \cite[Chapter 5, Theorem 3.1]{G-K-Z}. Indeed, denote by $T_0$ the open dense torus $\Spec(\CC[\Xi_{\alpha}^{\prime \prime}])\simeq (\CC^*)^{\d \Delta_{\alpha} -\d \Delta_{\beta}}$ of $Z_{\alpha}$. Then the open dense subset $W_{\alpha} \subset Y_{\alpha}$ is a direct sum $T_0 \sqcup T_0 \sqcup \cdots \sqcup T_0$ of $l$ copies of $T_0$. For $k=0,1,\ldots, l-1$, consider also surjective homomorphisms
\begin{equation}
I_k^{\prime} \colon \CC[\Xi_{\alpha}/\Xi_{\beta}] \longtwoheadrightarrow \CC[\Xi_{\alpha}^{\prime \prime}] \simeq \CC[\Xi_{\alpha}/\Xi_{\alpha}^{\prime}]
\end{equation}
of $\CC$-algebras defined by
\begin{equation}
\dsum_{s_i \in \Xi_{\alpha}}a_i \cdot [s_i+\Xi_{\beta}] \longmapsto \dsum_{s_i \in \Xi_{\alpha}} a_i \cdot \mu_l^{ke(s_i)}\cdot [s_i+\Xi_{\alpha}^{\prime}].
\end{equation}
Then by this homomorphism $I_k^{\prime}$ we obtain a closed embedding
\begin{equation}
\iota_k^{\prime} \colon T_0 \longhookrightarrow W_{\alpha}
\end{equation}
which induces an isomorphism from $T_0$ to the ($k+1$)-th component of $W_{\alpha}$. Moreover $\iota_k^{\prime}$ fits into the commutative diagram
\begin{equation}
\xymatrix{
Z_{\alpha} \ar@{^{(}->}[r]^{\iota_k} & Y_{\alpha} \\
T_0 \ar@{_{(}->}[u]\ar@{^{(}->}[r]^{\iota_k^{\prime}}& W_{\alpha}. \ar@{_{(}->}[u]
}
\end{equation}
Then we have an isomorphism
\begin{equation}
\bigoplus_{k=0}^{l-1}(\iota_k)_*(\CC_{T_0})\simeq \CC_{W_{\alpha}}
\end{equation}
in $\Dbc(Y_{\alpha})$. Therefore, applying Theorem \ref{thm:pre-10} to $\iota_k$ ($k=0,1,2,\ldots,l-1$), we obtain
\begin{equation}
l_{\alpha, \beta}=\sum_{k=0}^{l-1}\chi(\psi_{g_k}(\CC_{T_0})_0),
\end{equation}
where we set $g_k:=g \circ \iota_k \in \CC[\SS_{\alpha}/\Delta_{\beta}]$ ($k=0,1,2,\ldots,l-1$). Finally by \cite[Corollary 3.6]{M-T-new2} (see also the proof of \cite[Chapter 10, Theorem 2.12]{G-K-Z}) we get
\begin{equation}
l_{\alpha, \beta}=(-1)^{\d \Delta_{\alpha}-\d \Delta_{\beta}-1}i(\Delta_{\alpha}, \Delta_{\beta}) \cdot u(\SS_{\alpha}/\Delta_{\beta}).
\end{equation}
This completes the proof.
\end{proof}

\section{Characteristic cycles of constructible sheaves}\label{sec:7}
In this section, we give a formula for the characteristic cycles of $T$-invariant constructible sheaves (see Definition \ref{newdef} below) on toric varieties and apply it to GKZ hypergeometric systems and intersection cohomology complexes.

First, let $X$ be a (not necessarily normal) toric variety over $\CC$ and $T\subset X$ the open dense torus which acts on $X$ itself. Let $X=\bigsqcup_{\alpha}X_{\alpha}$ be the decomposition of $X$ into $T$-orbits.

\begin{definition}\label{newdef}
\begin{enumerate}
\item We say that a constructible sheaf $\F$ on $X$ is $T$-invariant if $\F|_{X_{\alpha}}$ is a locally constant sheaf of finite rank for any $\alpha$.
\item We say that a constructible object $\F \in \Dbc(X)$ is $T$-invariant if the cohomology sheaf $H^j(\F)$ of $\F$ is $T$-invariant for any $j \in \ZZ$.
\end{enumerate}
\end{definition}

Note that the so-called $T$-equivariant constructible sheaves on $X$ are $T$-invariant in the above sense. Recall also that to any object $\F$ of $\Dbc(X)$ we can associate a $\ZZ$-valued constructible function $\rho(\F) \in \CF_{\ZZ}(X)$ defined by
\begin{equation}
\rho(\F)(x)=\dsum_{j \in \ZZ} (-1)^j \d_{\CC}H^j(\F)_x  \hspace{5mm}(x \in X).
\end{equation}
If moreover $\F$ is $T$-invariant, clearly $\rho(\F)$ is constant on each $T$-orbit $X_{\alpha}$. In this case, we denote the value of $\rho(\F)$ on $X_{\alpha}$ by $\rho(\F)_{\alpha} \in \ZZ$. By using the fact that vanishing and nearby cycle functors send distinguished triangles to distinguished triangles, we can easily prove the following.

\begin{proposition}\label{prp:7-2}
Let $f \colon X \longrightarrow \CC$ be a non-constant regular function on the toric variety $X$ and set $X_0=\{x\in X\ |\ f(x)=0\} \subset X$. Then for any $T$-invariant object $\F\in \Dbc(X)$ and $x \in X_0$ we have
\begin{gather}
\chi(\psi_f(\F)_x)=\dsum_{\alpha}\rho(\F)_{\alpha}\cdot\chi(\psi_f(\CC_{X_{\alpha}})_x),\\
\chi(\varphi_f(\F)_x)=\dsum_{\alpha}\rho(\F)_{\alpha}\cdot\chi(\varphi_f(\CC_{X_{\alpha}})_x).
\end{gather}
\end{proposition}

Now let $X \longhookrightarrow Z$ be a closed embedding of the toric variety $X$ into a smooth algebraic variety $Z$ and $\F\in\Dbc(X)$ a $T$-invariant object. We consider $\F$ as an object in $\Dbc(Z)$ by this embedding and denote by $CC(\F)$ its characteristic cycle in the cotangent bundle $T^*Z$. Then there exist some integers $m_{\alpha} \in \ZZ$ such that
\begin{equation}
CC(\F)=\dsum_{\alpha}m_{\alpha} \left[\overline{T_{X_{\alpha}}^*Z}\right]
\end{equation}
in $T^*Z$. It is well-known that the coefficients $m_{\alpha}$ satisfy the formula
\begin{equation}
\rho(\F)=\dsum_{\alpha} (-1)^{\d X_{\alpha}}m_{\alpha} \cdot \Eu_{ \overline{X_{\alpha}} }.
\end{equation}
Moreover $m_{\alpha}$ are uniquely determined by this formula. Since the calculation of the Euler obstructions $\Eu_{\overline{X_{\alpha}}}$ does not depend on the choice of the embedding $X \longhookrightarrow Z$ (see \cite{Kashiwara}), the coefficients $m_{\alpha}$ do not depend on the choice of the smooth ambient space $Z$.

Now let $N \simeq \ZZ^n$ be a $\ZZ$-lattice of rank $n$ and $\sigma$ a strongly convex rational polyhedral cone in $N_{\RR}$. We take the dual $\ZZ$-lattice $M$ of $N$ and consider the polar cone $\sigma^{\vee}$ of $\sigma$ in $M_{\RR}$ as before. Then $X=\Spec(\CC[\sigma^{\vee} \cap M])$ is a normal toric variety and its open dense torus $T$ is $\Spec(\CC[M])$. We denote by $X_{\alpha}$ the $T$-orbit which corresponds to a face $\Delta_{\alpha}$ of $\sigma^{\vee}$ and consider the decomposition $X=\bigsqcup_{\Delta_{\alpha} \prec \sigma^{\vee}} X_{\alpha}$ of $X$ into $T$-orbits. In this situation, we have the following result.

\begin{theorem}\label{thm:7-3}
Let $X \longhookrightarrow Z$ be a closed embedding of $X$ into a smooth algebraic variety $Z$ and $\F \in \Dbc(X)$ a $T$-invariant object. Then the coefficients $m_{\beta} \in \ZZ$ in the characteristic cycle
\begin{equation}
CC(\F)=\dsum_{\Delta_{\beta} \prec \sigma^{\vee}}m_{\beta} \left[\overline{T_{X_{\beta}}^*Z} \right]
\end{equation}
are given by the formula
\begin{equation}\label{eq-1}
m_{\beta}=\dsum_{\Delta_{\beta} \prec \Delta_{\alpha} \prec \sigma^{\vee} }(-1)^{\d \Delta_{\alpha}} \rho(\F)_{\alpha}\cdot \RSV_{\ZZ}(\Delta_{\alpha}, \Delta_{\beta})
\end{equation}
(for the definition of the normalized relative subdiagram volume $\RSV_{\ZZ}(\Delta_{\alpha}, \Delta_{\beta})$ see Definition \ref{dfn:4-2}).
\end{theorem}

\begin{proof}
Since the coefficients of the characteristic cycle $CC(\F)$ are calculated by vanishing cycles as we explained in Section \ref{sec:pre}, by Proposition \ref{prp:7-2} we have
\begin{equation}
CC(\F)=\dsum_{\Delta_{\alpha} \prec \sigma^{\vee}}\rho(\F)_{\alpha} \cdot CC(\CC_{X_{\alpha}})
\end{equation}
in $T^*Z$. For a face $\Delta_{\beta} \prec \sigma^{\vee}$ of $\sigma^{\vee}$, we will show \eqref{eq-1}. It is enough to prove that for any face $\Delta_{\alpha} \prec \sigma^{\vee}$ of $\sigma^{\vee}$ such that $\Delta_{\beta} \prec \Delta_{\alpha}$ the coefficient $m_{\alpha, \beta} \in \ZZ$ of $\left[\overline{T_{X_{\beta}}^*Z} \right]$ in the characteristic cycle $CC(\CC_{X_{\alpha}})$ of $\CC_{X_{\alpha}} \in \Dbc(Z)$ is given by $m_{\alpha, \beta}= (-1)^{\d \Delta_{\alpha}} \RSV_{\ZZ}(\Delta_{\alpha}, \Delta_{\beta})$. Since in the case $\Delta_{\beta}=\Delta_{\alpha}$ we obtain it easily, it is enough to consider the case $\Delta_{\beta} \neq \Delta_{\alpha}$. From now on, we shall inherit and freely use the notations in the proof of Theorem \ref{thm:4-3}. In particular, in a neighborhood of $X_{\beta}=T_{\beta}$ in $X$ we have
\begin{gather}
\overline{X_{\alpha}}= X_{\alpha, \beta} \times T_{\beta},\\
X=X_{\sigma, \beta} \times T_{\beta}
\end{gather}
and there exists a unique point $q \in X_{\sigma, \beta}$ such that $\{ q\} \times T_{\beta}=T_{\beta}$. Let us take a closed embedding $X_{\sigma, \beta} \longhookrightarrow \CC^m$ by which the point $q \in X_{\sigma, \beta}$ is sent to $0 \in \CC^m$. Since the coefficient $m_{\alpha, \beta}$ in the characteristic cycle $CC(\CC_{X_{\alpha}})$ is independent of the choice of the ambient manifold $Z$, we may replace $Z$ by $Z^{\prime}:=\CC^m \times T_{\beta}$ and compute it in $Z^{\prime}$. Since $X_{\alpha}=T_{\alpha} \simeq T_{\alpha, \beta} \times X_{\beta}$, we obtain an isomorphism
\begin{equation}
\CC_{X_{\alpha}} \simeq (\CC_{T_{\alpha,\beta}}[-\d \Delta_{\beta}])\boxtimes(\CC_{X_{\beta}} [\d \Delta_{\beta}])
\end{equation}
in $\Dbc(Z^{\prime})$. Hence we get
\begin{eqnarray}
CC(\CC_{X_{\alpha}})
&=& CC(\CC_{T_{\alpha,\beta}}[-\d \Delta_{\beta}])\times CC(\CC_{X_{\beta}}[\d \Delta_{\beta}])\\
&=& CC(\CC_{T_{\alpha,\beta}}[-\d \Delta_{\beta}])\times [T^*_{X_{\beta}}X_{\beta}]
\end{eqnarray}
in $T^*Z^{\prime}=T^*(\CC^m) \times T^*X_{\beta}$. Since we have $\overline{T_{X_{\beta}}^*Z^{\prime}} =T^*_{\{0\}}(\CC^m)\times T^*_{X_{\beta}}X_{\beta}$, $m_{\alpha,\beta}$ is equal to the coefficient of $\left[T^*_{\{0\}}(\CC^m)\right]$ in the characteristic cycle $CC(\CC_{T_{\alpha, \beta}}[-\d \Delta_{\beta}])$ of $\CC_{T_{\alpha, \beta}}[-\d \Delta_{\beta}] \in \Dbc(\CC^m)$. Hence by taking a generic linear form $f \colon \CC^m \longrightarrow \CC$ we have
\begin{eqnarray}
m_{\alpha, \beta}
&=& -\chi(\varphi_f(\CC_{T_{\alpha,\beta}}[-\d \Delta_{\beta}])_0)\\
&=& (-1)^{\d \Delta_{\beta}+1}\chi(\varphi_f(\CC_{T_{\alpha, \beta}})_0).
\end{eqnarray}
By applying Theorem \ref{thm:pre-10} to the closed embedding $X_{\alpha, \beta} \longhookrightarrow \CC^m$ we obtain
\begin{equation}
m_{\alpha, \beta}= (-1)^{\d \Delta_{\beta}+1} \chi(\varphi_g(\CC_{T_{\alpha,\beta}})_0),
\end{equation}
where we set $g:=f|_{X_{\alpha,\beta}}$. Note that if $\Delta_{\beta} \subsetneq \Delta_{\alpha}$ the stalk of $\CC_{T_{\alpha,\beta}}$ at $0 \in X_{\alpha,\beta}$ is zero and
\begin{equation}
\chi(\varphi_g(\CC_{T_{\alpha,\beta}})_0)=\chi(\psi_g(\CC_{T_{\alpha,\beta}})_0).
\end{equation}
Finally by \cite[Corollary 3.6]{M-T-new2} (see also the proof of \cite[Chapter 10, Theorem 2.12]{G-K-Z}) we obtain the 
desired formula
\begin{equation}
m_{\alpha, \beta}= (-1)^{\d \Delta_{\alpha}} \RSV_{\ZZ}(\Delta_{\alpha},\Delta_{\beta}).
\end{equation}
This completes the proof.
\end{proof}

By the proof of Theorem \ref{thm:6-7}, we can prove also a similar result for projective toric varieties associated with lattice points. Let $A$ be a finite subset of $M \simeq \ZZ^n$ such that the convex hull $P$ of $A$ in $M_{\RR}$ is $n$-dimensional. We inherit the notations in Section \ref{sec:2} and Section \ref{sec:4}. Let us consider the projective toric variety $X_A \subset Z=\PP^{\sharp A-1}$ associated with $A$. For a face $\Delta_{\alpha}$ of $P$, denote by $X_{\alpha}(=T_{\alpha})$ the $T$-orbit which corresponds to $\Delta_{\alpha}$. Then we obtain a decomposition $X_A=\bigsqcup_{\Delta_{\alpha} \prec P} X_{\alpha}$ of $X_A$ into $T$-orbits.

\begin{theorem}\label{thm:7-4}
Let $X_A \longhookrightarrow Z=\PP^{\sharp A-1}$ be the projective embedding of $X_A$ and $\F \in \Dbc(X_A)$ a $T$-invariant object. Then the coefficients $m_{\beta} \in \ZZ$ in the characteristic cycle
\begin{equation}
CC(\F)=\dsum_{\Delta_{\beta} \prec P}m_{\beta}\left[\overline{T_{X_{\beta}}^*Z} \right]
\end{equation}
are given by the formula
\begin{equation}
m_{\beta}=\dsum_{\Delta_{\beta} \prec \Delta_{\alpha} \prec P }(-1)^{\d \Delta_{\alpha}} \rho(\F)_{\alpha} \cdot i(\Delta_{\alpha}, \Delta_{\beta}) \cdot u(\SS_{\alpha}/\Delta_{\beta})
\end{equation}
(for the definitions of $i(\Delta_{\alpha}, \Delta_{\beta})$ and $u(\SS_{\alpha}/\Delta_{\beta})$ see Definition \ref{dfn:6-5} and \ref{dfn:6-6}).
\end{theorem}

Since the proof of this theorem is similar to that of Theorem \ref{thm:7-3}, we omit it.

\begin{example}
Assume that the finite set $A=\{ \alpha(1), \alpha(2), \ldots , \alpha(m+1)\} \subset \ZZ^n$ generates $M=\ZZ^n$. For $j=1,2, \ldots , m+1$, set $\tl{\alpha(j)}:=(\alpha(j),1) \in \ZZ^{n+1}$ and consider the $(n+1)\times (m+1)$ integer matrix
\begin{equation}
\tl{A}:=\begin{pmatrix} {^t\tl{\alpha(1)}} & {^t\tl{\alpha(2)}} &\cdots& {^t\tl{\alpha(m+1)}} \end{pmatrix}=(a_{ij}) \in M(n+1, m+1; \ZZ)
\end{equation}
whose $j$-th column is ${^t\tl{\alpha(j)}}$. For $\gamma \in \CC^{n+1}$, we set 
\begin{eqnarray}
P_i&:=&\sum_{j=1}^{m+1} a_{ij}x_j\dfrac{\partial}{\partial x_j}-\gamma_i \hspace{5mm}(1\leq i\leq n+1),\\
\square_b&:=&\prod_{b_j>0}\left(\dfrac{\partial}{\partial x_j}\right)^{b_j} -\prod_{b_j<0} \left(\dfrac{\partial}{\partial x_j}\right)^{-b_j}\hspace{5mm} (b\in {\rm Ker} \tl{A} \cap \ZZ^{m+1}).
\end{eqnarray}
Then the GKZ hypergeometric system on $\CC_x^{m+1}$ associated with $\tl{A}$ and the parameter $\gamma \in \CC^{n+1}$ is
\begin{equation}
\begin{cases}
P_if(x)=0 &(1\leq i\leq n+1), \\
\square_b f(x)=0 & (b\in {\rm Ker} \tl{A} \cap \ZZ^{m+1})
\end{cases}
\end{equation}
(see \cite{G-K-Z-1} and \cite{G-K-Z-2}). Let $\D_{\CC_x^{m+1}}$ be the sheaf of differential operators with holomorphic coefficients on $\CC_x^{m+1}$. Then the coherent $\D_{\CC_x^{m+1}}$-module
\begin{equation}
\M_{A, \gamma}:=\D_{\CC_x^{m+1}}\Bigg/\left(\sum_{1 \leq i \leq n+1} \D_{\CC_x^{m+1}} P_i + \sum_{b\in {\rm Ker} \tl{A} \cap \ZZ^{m+1}}\D_{\CC_x^{m+1}} \square_b\right)
\end{equation}
which corresponds to the above GKZ system is holonomic. Let $\CC^{m+1}_{\xi}$ be the dual vector space of $\CC^{m+1}_x$ and $\M_{A, \gamma}^{\wedge}$ the Fourier transform of $\M_{A, \gamma}$ on $\CC^{m+1}_{\xi}$ (see \cite{H-T-T} etc.). We denote by $S_0$ the image of the map $\Psi_{\tl{A}} \colon (\CC^*)^{n+1} \longrightarrow \CC^{m+1}_{\xi}$ defined by $\Psi_{\tl{A}}(y)=(y_1^{\tl{\alpha(1)}}, y_2^{\tl{\alpha(2)}}, \ldots , y_{m+1}^{\tl{\alpha(m+1)}})$ and let $\iota \colon S_0 \longhookrightarrow \CC^{m+1}_{\xi}$ be the inclusion. Then $S_0 \subset \CC^{m+1}_{\xi}$ is the cone over the open dense $T$-orbit in $X_A \subset \PP^m$. In \cite{G-K-Z-2}, for a local system $\L$ of rank one on $S_0$ Gelfand-Kapranov-Zelevinsky constructed a morphism
\begin{equation}\label{eq:5-26}
j_! \L [n+1] \longrightarrow \Rhom_{\D_{\CC_{\xi}^{m+1}}}(\M_{A,\gamma}^{\wedge}, \O_{\CC_{\xi}^{m+1}})[m+1]
\end{equation}
in ${\rm Perv}(\CC^{m+1}_{\xi})$. In \cite[Theorem 4.6]{G-K-Z-2}, by calculating the characteristic cycles of both sides of \eqref{eq:5-26}, they proved that \eqref{eq:5-26} is an isomorphism if $\gamma \in \CC^{n+1}$ is generic (non-resonant in the sense of \cite[Theorem 4.6]{G-K-Z-2}). For each face $\Delta \prec P$ let $V_0(\Delta) \subset (\PP^m)^*$ be the dual variety of the closure $\overline{T_{\Delta}} \subset \PP^m$ of the $T$-orbit in $X_A$ which corresponds to $\Delta$ and denote by $V(\Delta) \subset \CC_x^{m+1}$ the cone over $V_0(\Delta) \subset (\PP^m)^*$. Then by Theorem \ref{thm:7-4} we have
\begin{equation}
CC( j_! \L [n+1] ) =\Vol_{\ZZ}(P)\left[T^*_{\CC_x^{m+1}} \CC_x^{m+1} \right] + \sum_{\Delta \prec P}i(P, \Delta) \cdot u(\Delta) \left[\overline{T_{V(\Delta)_{\reg}}^* \CC_x^{m+1} } \right]
\end{equation}
in $T^* \CC^{m+1}_{\xi} \simeq T^* \CC^{m+1}_{x}$, where for $\Delta_{\beta}=\Delta \prec \Delta_{\alpha}=P$ we set $u(\SS_{\alpha}/\Delta_{\beta})=: u(\Delta) \in \ZZ_{\geq 1}$. Moreover if $\gamma \in \CC^{n+1}$ is non-resonant, by the proof of \cite[Theorem 5]{G-K-Z-1} and \eqref{eq:4-53} we can show that the characteristic cycle $CC(\M_{A, \gamma})=CC(\M_{A, \gamma}^{\wedge})$ $=CC(\Rhom_{\D_{\CC_{\xi}^{m+1}}}(\M_{A, \gamma}^{\wedge}, \O_{\CC_{\xi}^{m+1}})[m+1])$ has the same expression. It seems that the integers $i(P, \Delta)$ are forgotten in \cite[Theorem 5]{G-K-Z-1}. Recently in \cite[Theorem 4.21]{S-W} Schulze and Walther proved it in the wider case where $\gamma$ is not rank-jumping.
\end{example}

From now on, we shall apply Theorem \ref{thm:7-3} to the intersection cohomology complexes on projective toric varieties. Let $M \simeq \ZZ^n$ be a $\ZZ$-lattice of rank $n$ and $N$ its dual lattice. Let $P$ be an integral polytope in $M_{\RR}$ such that $\d P=n=\d M_{\RR}$ and $\Sigma_P$ its normal fan in $N_{\RR}$ (see Subsection \ref{sec:4-2}). Denote by $X_{\Sigma_P}$ the (normal) toric variety associated with $\Sigma_P$. Then by \cite[Theorem 2.13]{Oda}, if $P$ is sufficiently large and $A=P \cap M$, the natural morphism $\Phi_A \colon X_{\Sigma_P} \longrightarrow \PP^{\sharp A -1}$ induces an isomorphism $X_{\Sigma_P} \simto X_A$. Let us consider the intersection cohomology complex $IC_{X_A} \in \Dbc(X_A)$ of such a projective toric variety $X_A \simeq X_{\Sigma_P} \subset \PP^{\sharp A-1}$. For simplicity, we set $X:=X_A$ and $Z:=\PP^{\sharp A-1}$. For a face $\Delta_{\alpha} \prec P$ of $P$, denote by $X_{\alpha}$ the $T$-orbit in $X$ which corresponds to $\Delta_{\alpha}$. Then $\F=IC_X[n] \in \Dbc(X)$ is a $T$-equivariant perverse sheaf on $X$. Considering $\F$ as a perverse sheaf on $Z=\PP^{\sharp A-1}$ via the embedding $X \longhookrightarrow Z$, we obtain the following results. For $\Delta_{\beta} \prec \Delta_{\alpha} \prec P$, we set $V_{\alpha,\beta}:=\RSV_{\ZZ}(\Delta_{\alpha}, \Delta_{\beta})$ and $V_{P,\beta}:=\RSV_{\ZZ}(P,\Delta_{\beta})$ for short.

\begin{example}\label{exa:5-6}
For $n=2, 3, 4$ the characteristic cycle of $\F=IC_X[n]\in {\rm Perv}(Z)$ in $T^*Z$ is given by
\begin{enumerate}
\item $n=2$:
\begin{equation}
CC(\F)=[T_T^*Z]+\dsum_{\begin{subarray}{c} \Delta_{\beta} \prec P\\ \d \Delta_{\beta} =0\end{subarray}} ( V_{P,\beta}-1)[T_{X_{\beta}}^*Z],
\end{equation}
\item $n=3$:
\begin{eqnarray}
\lefteqn{CC(\F)}\nonumber \\
&=&[T_T^*Z] +\dsum_{\begin{subarray}{c} \Delta_{\beta} \prec P\\ \d \Delta_{\beta}=1 \end{subarray}} (V_{P, \beta}-1) [T_{X_{\beta}}^*Z] + \dsum_{\begin{subarray}{c} \Delta_{\beta} \prec P \\ \d \Delta_{\beta}=0\end{subarray}} \left\{ V_{P,\beta}- \dsum_{\begin{subarray}{c} \Delta_{\beta} \prec \Delta_{\alpha} \prec P\\ \d \Delta_{\alpha}=2 \end{subarray}} V_{\alpha,\beta} +2\right\}[T_{X_{\beta}}^*Z],\nonumber\\
\end{eqnarray}
\item $n=4$:
\begin{eqnarray}
\lefteqn{CC(\F)}\nonumber \\
&=&[T_T^*Z] +\dsum_{\begin{subarray}{c} \Delta_{\beta} \prec P\\ \d \Delta_{\beta}=2 \end{subarray}} (V_{P,\beta}-1) [T_{X_{\beta}}^*Z] + \dsum_{\begin{subarray}{c} \Delta_{\beta} \prec P \\ \d \Delta_{\beta}=1\end{subarray}} \left\{ V_{P, \beta} -\hspace*{-2mm}\dsum_{\begin{subarray}{c} \Delta_{\beta} \prec \Delta_{\alpha} \prec P\\ \d \Delta_{\alpha}=3 \end{subarray}}V_{\alpha,\beta} +2\right\}[T_{X_{\beta}}^*Z]\nonumber \\
& &+\dsum_{\begin{subarray}{c} \Delta_{\beta} \prec P \\ \d \Delta_{\beta}=0\end{subarray}} \left\{ V_{P,\beta}-\hspace*{-2mm}\dsum_{\begin{subarray}{c} \Delta_{\beta} \prec \Delta_{\alpha} \prec P\\ \d \Delta_{\alpha}=3\end{subarray}} (V_{\alpha,\beta} +1)+\hspace*{-2mm}\dsum_{\begin{subarray}{c} \Delta_{\beta} \prec \Delta_{\alpha} \prec P\\ \d \Delta_{\alpha}=2\end{subarray}} V_{\alpha,\beta} +1\right\}[T_{X_{\beta}}^*Z].
\end{eqnarray}
\end{enumerate}
\end{example}

Note that (i) is a consequence of the main result of Gonzalez-Sprinberg \cite{G-S}. Also (i) and (ii) can be deduced from Theorem \ref{thm:7-3} and the combinatorial formula for the intersection cohomology complex $IC_X \in \Dbc(X)$ proved by Fieseler \cite{Fieseler} etc. We leave the proof to the reader. By (i) and a result of Gonzalez-Sprinberg \cite{G-S}, we obtain the following.

\begin{corollary}\label{cor:7-8}
If $n=2$, then the following three conditions are equivalent.
\begin{enumerate}
\item $X=X_A \simeq X_{\Sigma_P}$ is smooth.
\item $\Eu_X\equiv 1$ on $X$.
\item The characteristic cycle $CC(\F)$ of $\F=IC_X[n]$ is irreducible.
\end{enumerate}
\end{corollary}

Motivated by our calculations in the dimensions $n=2,3$ and $4$, we conjecture that the same equivalence would hold also for higher dimensions $n \geq 3$.

\bibliographystyle{amsplain}

\end{document}